\documentclass{amsart}
\usepackage{ae,lmodern} 
\usepackage{graphicx}
\usepackage{mathtools}
\usepackage{multicol}
\usepackage{amsmath}
\usepackage{amsfonts}
\usepackage{amssymb}
\usepackage{color}
\usepackage{tabularx}
\usepackage[left=3.25cm,right=3.25cm,top=2.5cm,bottom=2.5cm]{geometry}
\usepackage{multicol}
\usepackage{url}
\usepackage{tikz}
\usepackage{enumitem}
\usepackage{setspace}
\usepackage{tikz-cd}
\usetikzlibrary{matrix}
\usetikzlibrary{arrows}
\usepackage{pgf,tikz,pgfplots}
\pgfplotsset{compat=1.15}
\usepackage{mathrsfs}
\usetikzlibrary{arrows}

\makeatletter
\@namedef{subjclassname@2020}{%
  \textup{2020} Mathematics Subject Classification}
\makeatother

\theoremstyle{plain}
\newtheorem{theorem}{Theorem}[section]
\newtheorem{corollary}[theorem]{Corollary}
\newtheorem{lemma}[theorem]{Lemma}
\newtheorem{proposition}[theorem]{Proposition}
\newtheorem{theoremAH}{Theorem AH}
\newtheorem{theoremA}{Theorem}
\newtheorem{propositionA}[theoremA]{Proposition}

\theoremstyle{definition}
\newtheorem{definition}[theorem]{Definition}
\newtheorem*{definition*}{Definition}

\newtheorem{example}[theorem]{Example}
\newtheorem{counterexample}[theorem]{Counter-example}

\newtheorem{remark}[theorem]{Remark}

\newcommand{\AC}{\mathbb{A}_{\mathbb{C}}}

\newcommand{\N}{\mathbb{N}}
\newcommand{\Z}{\mathbb{Z}}
\newcommand{\Q}{\mathbb{Q}}
\newcommand{\R}{\mathbb{R}}
\newcommand{\K}{\mathbb{K}}
\renewcommand{\L}{\mathbb{L}}
\newcommand{\C}{\mathbb{C}}
\newcommand{\s}{\mathbb{S}} 
\newcommand{\D}{\mathcal{D}}
\newcommand{\T}{\mathbb{T}}
\newcommand{\GC}{\mathbb{G}_{m,\mathbb{C}}}

\newcommand{\GR}{\mathbb{G}_{m,\mathbb{R}}} 
\newcommand{\Spec}{\mathrm{Spec}} 
 
\newcommand{\Frac}{\mathrm{Frac}} 
\renewcommand{\H}{\mathrm{H}} 
\newcommand{\Aut}{\mathrm{Aut}} 
\newcommand{\Hom}{\mathrm{Hom}} 
\newcommand{\Div}{\mathrm{div}} 
\newcommand{\Gal}{\mathrm{Gal}} 
\newcommand{\GL}{\mathrm{GL}} 
\newcommand{\RC}{{R}_{\mathbb{C} /  \mathbb{R}}} 
\newcommand{\PPDiv}{\mathrm{PPDiv}_{\Q}}

\title{Real torus actions on real affine algebraic varieties}
 
\author{Pierre-Alexandre Gillard}

\thanks{The IMB receives support from  the EIPHI Graduate School (contract ANR-17-EURE-0002).
}

\address{Institut de Math\'{e}matiques de Bourgogne, UMR 5584 CNRS, Universit\'{e} Bourgogne Franche-Comt\'{e}, F-21000 Dijon, France}
\email{pierre-alexandre.gillard@u-bourgogne.fr}

\keywords{Affine variety, torus action, real variety,  real structure, real form}

\subjclass[2020]{%
  14R20
, 14L30
, 20G20
, 11E72
, 14P99
}

\begin{document}

\begin{abstract}
We extend the Altmann-Hausen presentation of normal affine algebraic $\C$-varieties endowed with effective torus actions to the real setting. In particular, we focus on actions of quasi-split real tori, in which case   we obtain a simpler presentation.
\end{abstract}

\maketitle

\tableofcontents

\section*{Introduction}

In the work of Altmann and Hausen in \cite{Alt}, normal affine algebraic  varieties endowed with effective torus actions over an algebraically closed field of characteristic zero are determined by  a   geometrico-combinatorial datum on a certain rational quotient for the action. This geometrico-combinatorial  presentation extends \textsl{mutatis mutandis} to actions of real split tori $\GR^n$ on normal affine algebraic $\R$-varieties.

In contrast, for normal   $\R$-varieties endowed with actions of a non-split torus $T$,  much less is known regarding the existence of a presentation similar to the split case.  However, this presentation  was   extended by Langlois in  \cite{Lang} for some complexity one\footnote{That is, effective actions of a  torus $T$ such that $\text{dim}(X)=\text{dim}(T)+1$} non-split torus actions     on normal affine  varieties $X$ over an arbitrary field.  This extension is based on a Galois descent construction specific to complexity one torus actions. On the other hand, the case where $T$ is the real circle $\mathbb{S}^1$ (of dimension 1) was studied by Dubouloz, Liendo and Petitjean in \cite{Lien, Petit}.   They gave  a  complete description of $\s^1$-actions on normal   affine $\R$-varieties based on the Altmann-Hausen presentation and on a Galois descent construction $\C/\R$ specific to $\s^1$-actions, with no restriction on the complexity.

In view of these results, it is natural and reasonable to expect that a general presentation of normal affine varieties endowed with torus actions over arbitrary fields of characteristic zero can be obtained by combining Altmann-Hausen theory for split torus actions with appropriate Galois descent methods. In this context, we give a complete description of real torus actions on normal $\R$-varieties. The Weil restriction $\RC(\GC)$ of $\GC$ is a real non-split torus (of dimension 2), and all real tori are isomorphic to a product of the three \textsl{elementary} real tori $\GR$, $\mathbb{S}^1$ and $\RC(\GC)$. We treat the missing case of $\RC(\GC)$-actions and more generally we  extend  the setting  of Altmann-Hausen to   real torus actions on normal  affine $\R$-varieties. We will pay a special attention to  actions of quasi-split tori, that is real tori with no $\s^1$-factors.

\medskip

In view of extending the Altmann-Hausen presentation to the real setting, we use the language of $\R$-structures on    algebraic $\C$-varieties.  An $\R$-structure on an algebraic $\C$-variety $X$ is an   involution of $\R$-schemes $\sigma$ on $X$ such that the following diagram commutes:
\begin{center}
\begin{tikzpicture}
  \matrix (m) [matrix of math nodes,row sep=1em,column sep=4em,minimum width=2em]
  {          X       &                          &     X                      \\
    \text{Spec}(\C) &                          &    \text{Spec}(\C)    \\
 };
  \path[-stealth]
    (m-1-1) edge node [above] {$\sigma    $} (m-1-3)
           edge  (m-2-1)
    (m-2-1) edge node [above] { \small{$\Spec(  z \mapsto \bar{z} )$} } (m-2-3)
		(m-1-3) edge  (m-2-3);
\end{tikzpicture}
\end{center}
An $\R$-morphism between two $\C$-varieties $X$ and $X'$ endowed with $\R$-structures $\sigma$ and $\sigma'$ is a morphism of $\C$-varieties $f : X \to X'$ such that $\sigma' \circ f = f \circ \sigma$. 
An $\R$-group structure $\tau$ on a complex algebraic group $G$ is an $\R$-structure on $G$ such that the multiplication $G \times G \to G$, the inverse $G \to G$ and the unity $\Spec(\C) \to G$ are $\R$-morphisms (see \S \ref{SectionGDAT} for details).  Let us note  that an  $\R$-group structure $\tau$ on a complex torus $\mathbb{T}$ corresponds to a lattice involution $\tilde{\tau}$ on its character lattice $M:= \Hom_{gr}(\mathbb{T}, \GC)$. 
 
There is an equivalence of categories between the category of quasi-projective algebraic $\R$-varieties (resp. real algebraic groups) and the category of quasi-projective algebraic $\C$-varieties endowed with an $\R$-structure (resp. complex algebraic groups endowed with an $\R$-group structure); see Proposition \ref{eqcat} for the precise statement. Therefore we will often  write $(X, \sigma)$ to refer to an algebraic $\R$-variety and  $(G, \tau)$ to refer to a real algebraic group.

\medskip

We now briefly explain Altman and Hausen's theory in order to state our main results. Let $\T$ be an $n$-dimensional complex torus with character lattice $M$. Then any algebraic action of $\T$ on an affine $\C$-variety $X$ corresponds to an $M$-grading $\C[X] = \bigoplus_{m \in M} \C[X]_m $ of its coordinate ring, the spaces $\C[X]_m$ consisting of semi-invariant regular functions of weight $m$ on $X$.  Let $\omega$ be a   full dimensional cone in $M_{\Q}:= M \otimes_{\Z} \Q$, let   $Y$ be a normal \textsl{semi-projective} variety (see Definition \ref{semiproj}), and let $\D:= \sum \Delta_i \otimes D_i$ be a  \textsl{proper polyhedral divisor}. This means that the $D_i$ are prime divisors on $Y$ and the coefficients $\Delta_i$ are convex polyhedra in   $N_{\Q}$ having $\omega^{\vee}$ as tail cone, where $N$ is the cocharacter lattice  (see Definition \ref{ppDiv}).  Then, for every $m \in \omega \cap M$, we can evaluate $\D$ in $m$ to obtain a Weil $\Q$-divisor $\D(m):= \sum \text{min}\{ \langle m | \Delta_i \rangle \} \otimes D_i$. From the datum ($\omega$, $Y$, $\D$),  Altmann and Hausen construct   an $M$-graded $\C$-algebra:
\begin{equation*}
A[Y, \D]:= \bigoplus_{m \in \omega \cap M} \H^0(Y, \mathcal{O}_Y(\D(m))) \subset \C(Y)[M].
\end{equation*}  
The main results of \cite{Alt} can be summarized as follows (see \S 3 for details):
\begin{theoremAH} \label{thmAltintro}\cite[Theorem 3.1]{Alt}. \textsl{The affine scheme $X[Y, \D] := \Spec(A[Y, \D])$  is a normal $\C$-variety  endowed with an effective  $\mathbb{T}$-action.}
\end{theoremAH}
\begin{theoremAH} \label{thmAltintro2}\cite[Theorem 3.4]{Alt}. \textsl{Let $X$ be an affine normal variety endowed with an effective $\mathbb{T}$-action. There exists a datum  $(\omega, Y, \D)$   such that the graded $\C$-algebras $\C[X]$ and $A[Y, \D]$ are isomorphic.}
\end{theoremAH}

\medskip

As mentioned above, the present article focuses on real torus actions on normal affine $\R$-varieties.  Our main results, Theorem \ref{ResultTintro} and Theorem \ref{ResultT2intro}, give a presentation of real torus actions in the language of \cite{Alt} extended to affine $\C$-varieties with $\R$-structures.

Let   $(\mathbb{T}, \tau)$ be a real torus, let $M$ be the  character lattice of $\mathbb{T}$, and let $\omega$ be a full dimensional cone in $M_{\Q}$.   Let $(Y, \sigma_Y)$ be a    semi-projective algebraic $\R$-variety and  let $\D$ be a proper polyhedral divisor on $Y$.  The first main result gives a condition on $\D$ for the existence of an $\R$-structure on the affine $\C$-variety  $X[Y,\D]$. This result is the real analog of Theorem AH \ref{thmAltintro}:  
\begin{theoremA}[Theorem \ref{ResultT}]  \label{ResultTintro}   
\textsl{If there exists a  {monoid}   {morphism} $h :   \omega \cap M \to \C(Y)^*$ such that 
\begin{equation*}
\forall m \in \omega \cap M, \ \sigma_Y^*(\D(m)) = \D(\tilde{\tau}(m)) + \Div_Y(h(\tilde{\tau}(m))) \text{ \ \  and \ \ } h(m) \sigma_Y^{\sharp}(h(\tilde{\tau}(m)))=1, 
\end{equation*}
then there exists an $\R$-structure $\sigma_{X[Y, \D]}$ on the normal affine variety $X{[Y, \D]} $ such that the real torus  $(\mathbb{T}, \tau)$ acts   on  the $\R$-variety $({X[Y, \D]}, \sigma_{X[Y, \D]})$.
}\end{theoremA}

Conversely, given a $\T$-action on an affine algebraic $\C$-variety $X$,  Altmann and Hausen give in \cite[\S 11]{Alt} a method   to construct a proper polyhedral divisor $\D$ on a semi-projective variety $Y$    based on the choice of an appropriate $\mathbb{T}$-equivariant  closed immersion  $X \hookrightarrow \AC^n$ and on the downgrading of the $\GC^n$-action on $\AC^n$ to a $\T$-action.  

 Thus, for  a $(\mathbb{T}, \tau)$-action on an affine algebraic $\R$-variety $(X, \sigma)$,  a key ingredient to construct an   $\R$-structure on the semi-projective variety $Y$ mentioned in Theorem AH \ref{thmAltintro2} is to find a certain $\mathbb{T}$-equivariant closed immersion $X \hookrightarrow \AC^n$ which is also $\Gal(\C/\R)$-equivariant:
\begin{propositionA}[Proposition  \ref{GaloisToricDowngrading}] \label{GaloisToricDowngradingintro} 
\textsl{There exists  $n \in \N,\  n \geq   \text{dim}(\mathbb{T})$,  such that the following hold:}
\begin{enumerate}[leftmargin=0.75cm, label=(\roman*)]
\item \textsl{There exists an  $\R$-group structure $\tau'$ on $\GC^n$   that extends to a  $\R$-structure $\sigma'$ on $\AC^n $;}
\item \textsl{$(\mathbb{T}, \tau)$ is a closed  {subgroup} of $(\GC^n, \tau')$;} and 
\item \textsl{$(X, \sigma)$ is a closed subvariety of   $(\AC^n, \sigma)$   and   $(X, \sigma) \hookrightarrow (\AC^n, \sigma')$ is $(\mathbb{T}, \tau)$-equivariant.} 
\end{enumerate} 
\end{propositionA}

The immersion $(\mathbb{T}, \tau) \hookrightarrow (\GC^n, \tau')$ induces an $\R$-group structure $\tau_Y$ on the  quotient torus $\mathbb{T}_Y := \GC^n/\mathbb{T}$.   This $\R$-group structure   $\tau_Y$ induces in turn an $\R$-structure $\sigma_Y$ on the semi-projective variety $Y$ mentioned in Theorem  AH \ref{thmAltintro2}. We  downgrade $\Gal(\C/\R)$-equivariantly   the $\GC^n$-action on $\AC^n$ to a $\T$-action,  which is a key ingredient in the proof of the following result (which is the real analogue of Theorem AH \ref{thmAltintro2}):
\begin{theoremA}[Theorem \ref{ResultT2}] \label{ResultT2intro}      
\textsl{Let  $ \omega \subset M_{\Q}$ be  the weight cone of the $\mathbb{T}$-action on $X$.    There exists  a   normal semi-projective $\R$-variety   $(Y, \sigma_Y)$,    a proper polyhedral divisor  $\D$   on $Y$, and  a  {monoid}  {morphism} $h :   \omega  \cap M \to \C(Y)^*$ such that  
$$ \forall  m \in \omega  \cap M, \ \sigma_Y^*(\D(m)) = \D(\tilde{\tau}(m)) + \Div_Y(h(\tilde{\tau}(m))) \text{ \ \  and \ \ } h(m) \sigma_Y^{\sharp}(h(\tilde{\tau}(m)))=1,$$
and such that there is an  isomorphism of $\R$-varieties between $(X, \sigma)$ and $({X[Y, \D]}, \sigma_{X[Y, \D]})$. } 
\end{theoremA}

\medskip

In the case where the real torus $T$ is   quasi-split,  our presentation simplifies.  Indeed,  if  $(X, \sigma)$ is endowed with a $T$-action and if $(Y, \sigma_Y)$ is the variety mentioned in Theorem \ref{ResultT2intro}, we see in Proposition  \ref{ResultQS} that there exists a proper polyhedral divisor $\D$ on $Y$ such that  $\sigma_Y^*(\D(m)) = \D(\tilde{\tau}(m))$ for all $m \in \omega  \cap M$; i.e we can \textit{take} $h=1$. From this result, we recover the Altmann-Hausen presentation for   $\GR^n$-actions.  On the other hand, this simplification is not always possible for $\mathbb{S}^1$-actions: see  \S \ref{SubsectionC} for details and examples. In this case we recover the presentation for $\mathbb{S}^1$-actions given by Dubouloz and Liendo in \cite{Lien}.

\medskip

After fixing our notation, the article is structured as follows. 

In \S \ref{SubsectionGD} we  recall well-known facts about $\R$-structures on $\C$-varieties,  and in \S \ref{SubsectionRT} we see that tori inclusions corresponds to certain short exact sequences of  lattices. Basic results on real torus actions and  examples   of real torus actions on   affine toric $\R$-varieties are  given in \S \ref{SubsectionRTA} and \S \ref{SubsectionATV}. 
 
In \S \ref{SectionCAHP}, we briefly explain Altmann-Hausen's theory in view of extending it to the real case. We start by introducing polyhedral divisors in \S \ref{SubsectionTPPD}, and we recall the main results of \cite{Alt} in \S \ref{SubsectionCAH}.  
 
In \S \ref{SectionRAHP},  after proving  Proposition \ref{GaloisToricDowngradingintro} in \S \ref{SubsectionETD}, we prove our main results in \S \ref{SubsectionRAH}: Theorems \ref{ResultTintro} and \ref{ResultT2intro}. Then,  we give several    cohomological results used to simply the Altmann-Hausen presentation in the case where the acting torus is quasi-split.

In \S \ref{SectionEx},  we give examples of  $\GR$-actions (see   \S \ref{SubsectionGR}), $\RC(\GC)$-actions (see   \S \ref{SubsectionWR}) and $\s^1$-actions (see   \S \ref{SubsectionC}).

\medskip

\medskip

\noindent \textbf{Acknowledgments.}
The author  is grateful to Charlie Petitjean for stimulating discussions  about   Altmann-Hausen theory on the field of complex numbers.  The author would like to thank the anonymous referee for her/his comments and suggestions that helped him to improve the quality of the article.

\bigskip

\section{Notation}

Throughout  the entire paper, we call a \textit{$\C$-variety} a separated integral scheme of finite type over $\C$, and an \textit{$\R$-variety}   a separated  geometrically integral scheme of finite type over  $\R$.  We denote by	$\Gamma:= \{ id, \gamma\}$ the Galois group of the field extension $\C/\R$, it is isomorphic to $\Z/2\Z$. 
The group of regular automorphisms of a $\C$-variety $X$ is denoted by $\Aut(X)$, and the group of regular group automorphisms of a complex algebraic group $G$ is denoted by $\Aut_{gr}(G)$. 

\smallskip

From here on,   $N$ denotes a lattice, i.e.~a finitely generated free abelian group,  and $M := \Hom_{\Z}(N, \Z)$ denotes its dual lattice. The associated $\Q$-vector spaces are denoted by $N_{\Q} := N \otimes_{\Z} \Q$ and $M_{\Q} := M \otimes_{\Z} \Q$ respectively, and the the corresponding pairing by: 
\begin{equation*}
 M \times N \to \Z, \ \ \ (u,v) \mapsto \langle u,v \rangle := u(v).
\end{equation*}
Let us recall some results of \cite[\S 1.2]{Fult}. Let $N'$ be  a  lattice,  and let $f : N \to N'$ be a lattice homomorphism. It induces a  unique $\Q$-linear map  $N_{\Q} \to N'_{\Q}$, also denoted by $f$. 
 A subset $\omega_N \subset N_{\Q} $ is called a \textit{convex polyhedral cone} if there exists  a finite set  $S \subset N_{\Q}$ such that 
\begin{equation*}
\omega_N  = \text{Cone}(S) := \left\{ \sum_{v \in S} \lambda_v v \ | \ \lambda_v \in {\Q}_{\geq 0} \right\} \subset N_{\Q}. 
\end{equation*}
A cone $\omega_N$ is   \textit{strongly convex}  if $\omega_N \cap (-\omega_N) = \{ 0 \}$. 
For us, a \textsl{cone} in $N_{\Q}$ is always a convex polyhedral cone.
 The dual cone of $\omega_N$ is defined by
\begin{equation*}
\omega_N^{\vee} := \{ u \in M_{\Q} \ | \  \forall v \in \omega_N, \ \langle u | v \rangle \geq 0  \};
\end{equation*}
it is a cone in $M_{\Q}$. Let  $\omega_N$ be a cone in $N_{\Q}$.   A \textit{face} $\tau_N$ of $\omega_N$ is given by $ \tau_N = \omega_N \cap u^{\perp}$, for some $u \in \omega_N^{\vee}$, where
$u^{\perp}    := \left\{ v \in \omega_N \ | \ \forall u \in \omega_N^{\vee},  \ \langle u,v \rangle  = 0    \right\} $. 
Recall that a face of a cone is a cone. 
The relative interior  $\text{Relint}(\omega_N )$ of a cone $\omega_N  $ is obtained  by removing all   proper   faces  from $\omega_N$. 

 \smallskip
   
A \textit{quasifan} $\Lambda$ in $ N_{\Q}$ (or in $M_{\Q}$) is a finite collection of cones in $ N_{\Q}$ (or in $M_{\Q}$)  such that,  for any $\lambda \in \Lambda$, all the faces   of $ \lambda$ belong to $\Lambda$,  and for any   $\lambda_1, \lambda_2 \in \Lambda$, the intersection $\lambda_1 \cap \lambda_2$ is a face of both $\lambda_i$. 
The \textit{support} of a quasifan is the union of all its cones. A quasifan is called a \textit{fan} if all its cones are strongly convex.

\smallskip

A complex torus is an affine algebraic group isomorphic to $\GC^n$. There is a one-to-one correspondence between lattices and complex tori. To a lattice $M \cong \Z^n$, we associate the affine variety $\Spec(\C[M])$, with $\C[M] := \{ \sum_{m \in M} c_m \chi^m \ | \   c_m \in \C \}$ and where $\chi^m $ are indeterminate such that $\chi^{m+m'}=\chi^m  \chi^{m'}$. It is a complex torus isomorphic to $\GC^n$. Conversely, to a complex torus $\T$ isomorphic to $\GC^n$, we associate its character lattice $M:= \Hom_{gr}(\T, \GC)$. It is isomorphic to $\Z^n$.  Let us  recall that $\Aut_{gr}(\GC^n) \cong \GL_{n}(\Z)$.

\smallskip

The action of  a complex torus $\T$ on a $\C$-variety $X$ is called \textit{effective} if the neutral element of $\T$ is the only element acting trivially on $X$. 
  In this paper, we only consider effective torus actions. Let $X$  be a $\C$-variety   endowed with an    action of the torus $\T=\Spec(\C[M])$. The \textit{weight} \textit{monoid} of this action is $S := \{ m \in M \ | \ \C[X]_m \neq \{0 \} \}$ and the cone $\omega_M$ of $M_{\Q}$ spanned by the weight monoid $S$ is called the \textit{weight} \textit\textit{cone}. The algebra $\C[X]$ is $M$-graded: $\C[X] = \bigoplus_{m \in \omega_M \cap M} \C[X]_m$. 
There is a bijective correspondence between the $\T$-actions on $X$ and the $M$-gradings on $\C[X]$  \cite[\S 2.1]{KamRuss}.

 \smallskip

We recall  some definitions and results useful for the proof of Lemma \ref{SectionGalois}. Let $(G,\cdot)$ be a group. A $G$-module is an abelian group $(M, +)$ endowed with an action $(g, m) \mapsto g \cdot m$ of $G$  such that the induced map $\varphi_{g} : m \mapsto g \cdot m$ is an abelian group automorphism. Recall that this data is equivalent to a left  module $(M,+)$ over the ring $\Z[G]$.  Indeed, if $M$ is a module over the ring $\Z[G]$, we define a $G$-module structure on $M$ via $g \cdot m :=  \chi^g m$ for all $(g, m) \in G \times M$. Conversely, if $M$ is a $G$-module, we construct a $\Z[G]$-module structure on $M$ via $(\sum_{g \in G} n_{g} \chi^g) m := \sum_{g \in G} n_{g}   g \cdot m$.

\section{Galois descent $\C / \R$ and algebraic tori} \label{SectionGDAT}

We  recall basic definitions and well-known facts about $\R$-structures on $\C$-varieties and $\R$-group structures on complex algebraic groups in view of studying torus actions on $\R$-varieties. 
See \cite[\S3.1.3]{Ben} and  \cite{BorelSerre}.

\subsection{Galois descent $\C / \R$} \label{SubsectionGD}

Let us briefly recall the classical correspondence between quasi-projective $\R$-varieties and quasi-projective $\C$-varieties endowed with an $\R$-structure. Every $\C$-variety $X$ can be viewed as an $\R$-scheme   via the composition of its structure morphism $ X \to \Spec(\C)$ with the morphism $\Spec(\C) \to \Spec(\R)$ induced by the inclusion $\R \hookrightarrow \C = \R[i]/(i^2+1)$. The Galois group $\Gamma$ acts on $\Spec(\C)$ by the usual complex conjugation $z \mapsto \overline{z}$.

\begin{definition}      
\begin{enumerate}[leftmargin=0.75cm, label=(\roman*)]
\item An  \textit{$\R$-form} of a $\C$-variety $X$ is an $\R$-variety $X_0$ together with an isomorphism $X_0 \times_{\Spec(\R)} \Spec(\C) \cong X$ of $\C$-varieties. By abuse of notation  we will often write: $X_0$ is an $\R$-form of $X$ instead of $(X_0,\cong)$. 
\item An \textit{$\R$-structure} $\sigma$ on a $\C$-variety $X$  is an antiregular involution, i.e, an involution of  $\R$-scheme    $\sigma : X \to X$ which makes the following diagram commute:

\begin{center}
\begin{tikzpicture}
  \matrix (m) [matrix of math nodes,row sep=1em,column sep=4em,minimum width=2em]
  {
          X       &                          &     X                      \\
    \text{Spec}(\C) &                          &    \text{Spec}(\C)    \\
};
  \path[-stealth]
    (m-1-1) edge node [above] {$\sigma    $} (m-1-3)
           edge  (m-2-1)
    (m-2-1) edge node [above] { \small{  $\Spec(  z \mapsto \bar{z} ) $}   } (m-2-3)
		(m-1-3) edge  (m-2-3);
\end{tikzpicture}
\end{center}
\item Two $\R$-structures $\sigma$ and $\sigma'$ on $X$ are \textit{equivalent} if there exists   $\varphi \in \Aut(X) $ such that $\sigma' = \varphi \circ \sigma  \circ \varphi^{-1}$.
\item An   \textit{$\R$-morphism} between two $\C$-varieties $X$ and $X'$ with $\R$-structures $ \sigma $ and $  \sigma' $ is a morphism of $\C$-varieties $f : X \to X'$ such that $\sigma' \circ f = f \circ \sigma$ as morphisms of $\R$-schemes.
\end{enumerate}
\end{definition}

If a quasi-projective $\C$-variety $  X $ is endowed with an $\R$-structure $\sigma$, then the quotient $  X/\langle \sigma \rangle $ exists in the category of $\R$-varieties and the structure morphism $  X \to \Spec(\C)$ descends to a morphism $X/\langle \sigma \rangle  \to \Spec(\R)$ making $X/\langle \sigma \rangle $ into an $\R$-variety   such that $X \cong  (X/\langle \sigma \rangle)_{\C} $.   
If $f : (X, \sigma  ) \to (X', \sigma' ) $ is an $\R$-morphism between quasi-projective $\C$-varieties, and if $\pi' : X' \to X'/\langle \sigma' \rangle$ denotes the quotient morphism, we obtain from the  invariant morphism $\pi' \circ f : X \to X'/\langle \sigma' \rangle$   a morphism $f_0 : X /\langle \sigma  \rangle \to  X'/\langle \sigma' \rangle$ of $\R$-varieties.

\begin{proposition} \label{eqcat} 
The functor $(X, \sigma) \mapsto   X/\langle \sigma \rangle$ induces an equivalence of categories between the  category of pairs $(X, \sigma)$ consisting of a   quasi-projective $\C$-variety $X$ endowed with an $\R$-structure $\sigma$ and the category of quasi-projective $\R$-varieties. Moreover, $\sigma$ is equivalent to $\sigma'$ if and only if $X/\langle \sigma \rangle$ is $\R$-isomorphic to $X/\langle \sigma'\rangle$. 
\end{proposition}

Using this equivalence, we often write $(X, \sigma)$ to refer to an algebraic $\R$-variety.

\begin{proof}
We give a sketch of the proof for the sake of completeness.  If  $(X_0, \cong)$ is an $\R$-form of $X$, the $\C$-variety 
$(X_0)_{\C} := X_0 \times_{\Spec(\R)} \Spec(\C)$   is endowed with a canonical $\R$-structure   given by the action of $\Gamma$ by complex conjugation on the second factor, this gives an $\R$-structure $\sigma$ on $X \cong (X_0)_{\C}$. 
If $(X_0, \cong)$ and $(X'_0, \cong)$ are $\R$-forms of $X$ and $X'$ respectively, and if $f_0 : X_0 \to X'_0$ is a morphism of $\R$-varieties, then $f_0 \times id : (X_0)_{\C} \to   (X'_0)_{\C}$ is a morphism of $\C$-varieties, so we obtain a  morphism $f : X \to X'$ such that $f \circ \sigma = \sigma' \circ f$.   
\end{proof}

We have similar definitions and properties for affine algebraic groups.  

\begin{definition} 
\begin{enumerate}[leftmargin=0.75cm, label=(\roman*)]
\item Let $G$ be a complex algebraic group. A real algebraic group $G_0$ together with an isomorphism $G_0 \times_{\Spec(\R)} \Spec(\C) \cong G$ is called an   \textit{$\R$-form} of $G$. 
\item  An  \textit{$\R$-group} \textit{structure} $\tau$ on a complex algebraic group $G$  is an $\R$-structure $\tau : G \to G$ such that the multiplication $G \times G \to G$, the inverse $G \to G$ and the unity $\Spec(\C) \to G$ are $\R$-morphisms. 
\item Two $\R$-group structures $\tau$ and $\tau'$ on $G$ are \textit{equivalent} if there exists   $\varphi \in \Aut_{gr}(G) $ such that $\tau' = \varphi \circ \tau \circ \varphi^{-1}$.  
\item An   \textit{$\R$-morphism} between two complex algebraic groups $G$ and $G'$ with $\R$-structures $ \tau $ and $  \tau' $ is a morphism of complex algebraic groups $f : G \to G'$ such that $\tau' \circ f = f \circ \tau$ as morphisms of $\R$-schemes.
\end{enumerate} 
\end{definition}

If $ G  $ is a complex affine algebraic group endowed with an $\R$-group structure $\tau$, then the quotient scheme $ G_0 := G/ \langle \tau \rangle$ is a real algebraic group which satisfies $(G_0)_{\C}  \cong G$ as complex algebraic groups.

\begin{remark}  
There is an equivalence between the category of pairs $(G, \tau)$ consisting of a complex affine algebraic group endowed with an $\R$-group structure, and  the category of real affine algebraic groups. This induces a one-to-one correspondence between the $\R$-forms of G, up to isomorphism in the category of real algebraic groups, and the equivalence classes of $\R$-group structures on G.
\end{remark}

\subsection{Real tori} \label{SubsectionRT}

We define real tori and we recall that  any real torus is isomorphic to a product of copies of three \textit{elementary} real tori.

\begin{definition}      
A \textit{real}  \textit{torus}  $T $ is a real affine algebraic group   such that $T_{\C} $ is a   complex torus.   It is called a \textit{split torus} if $T \cong \GR^n$  for some integer $n$.
\end{definition}

The torus $\GC$ has two non isomorphic $\R$-forms: the real split torus $\GR$ and the real circle $\mathbb{S}^1 := \Spec \left( {\mathbb{R}[x , y]}/{(x^2+y^2-1)} \right)$. Since $\Aut_{gr}(\GC)= \{ id, -id \} $, the equivalence class of an $\R$-group structure on $\GC$ has only one element. So,  {the}   $\R$-group structures on $\GC$ associated to $\GR$ and $\mathbb{S}^1$ are respectively:
\begin{equation*}
{\tau_0} :   
z   \mapsto  \overline{z} \ \ \ \text{ and } \ \ \ {\tau_1} :  
z   \mapsto  \overline{z}^{\ -1}.
\end{equation*}
The group structure on $\s^1$ is given by
$$(x,y) \cdot (x', y') =(xx' - yy', xy'+yx').$$
The \textit{{Weil} {restriction}} of $\GC$ is $\RC(\GC) := \Spec \left(  {\R[x_1,  y_1, x_2, y_2]}/{(x_1y_1-x_2y_2-1, x_2y_1+x_1y_2)} \right)$.  It is an $\R$-form of $\GC^2$.  An  $\R$-group  structure on $\GC^2$ associated to $\RC(\GC)$ is:
\begin{equation*}
{\tau}_{2}  :    (z,w)   \mapsto (  \overline{w} , \overline{z } ).
\end{equation*}
The group structure on $\RC(\GC)$ is given by
$$(x_1, y_1, x_2,  y_2) \cdot (x_1', y'_1, x'_2, y_2') =(x_1x_1' - x_2x'_2, y_1y_1'-y_2y_2', x_1x_2' + x_2x_1', y_1y_2'+y_2y_1').$$
By abuse, we call {Weil} {restriction} any real torus isomorphic to $\RC(\GC)$. Here, $\Aut_{gr}(\GC^2)  \cong \GL_2(\Z)$,  so the equivalence class of an $\R$-group structure on $\GC^2$ has infinitely many elements.  
 For instance   
\begin{equation*}
{\tau}_{2}'  :    (z,w)   \mapsto (  \overline{w}^{\ -1} , \overline{z }^{\ -1} )     \text{ and } 
{\tau}_{2}''  :   (z,w)   \mapsto (  \overline{z}^{\ -1} \overline{w } , \overline{w } )    
\end{equation*}
 are   $\R$-group structures equivalent to $\tau_2$.

\begin{remark} An $\R$-group structure $\tau$ on a complex torus $\T$  induces  lattices involutions $\tilde{\tau}$ and $\hat{\tau}$ on $M:= \Hom_{gr}(\T, \GC)$  and   $N:= \Hom_{gr}(  \GC, \T)$ respectively. For $(\GC^2, \tau_2)$,   the involutions $\tilde{\tau}_{2}$ and $\hat{\tau}_{2}$ are both given by  $\Z^2 \to \Z^2,  (k,l)   \mapsto (  l , k ) $.  For $\GR$, $\tilde{\tau}_{0}$ and $\hat{\tau}_{0}$ are both given by  $id : \Z  \to \Z$, and for $\mathbb{S}^1$, $\tilde{\tau}_{1}$ and  $\hat{\tau}_{1}$ are both given by  $-id : \Z  \to \Z$.
\end{remark}

These three \textit{elementary} real tori form the building blocks of every real torus, that is:

\begin{proposition} \cite[Proposition 1.5]{Moser}. \label{RealGroupStructureTorus}  
Every $\R$-group structure on   $\GC^n$ is equivalent to exactly one   $\R$-group structure of the form  $\tau_{0}^{\times n_0} \times \tau_{1}^{\times n_1} \times \tau_{2}^{\times n_{2}}$,   with   $n_0+n_1+2n_{2}=n$.
\end{proposition}

\begin{remark} \label{FracsectionR} Let $(\T , \tau)$ be a subtorus of the real torus $(\GC^n, \tau')$. Let $M:= \Hom_{gr}(\T , \GC)$ and $M':= \Hom_{gr}(\GC^n, \GC)$. The inclusion $\T \hookrightarrow \GC^n$ induces   a surjective lattice homomorphism $M' \to M$.
 Let  ${M}_Y$ be the kernel of this homomorphism, it is  a sublattice of $M'$.  Moreover, the lattice involution $\tilde{\tau}'$ on $M'$ induces a lattice involution $\tilde{\tau}_Y$ on $M_Y$. Let ${\tau}_Y$ be the induced $\R$-group structure on $\T_Y$.  The following  diagram of complex algebraic groups commutes:
\begin{center}
\begin{tikzpicture}
\matrix (m) [matrix of math nodes,row sep=0.7em,column sep=5em,minimum width=2em]  {
	1^{\vphantom{n}}_{\vphantom{n}}   & \T^{\vphantom{n}}_{\vphantom{n}}   & \GC^n & {\T_Y}^{\vphantom{n}}_{\vphantom{n}} := \Spec(\C[M_Y]) & 1    \\
 	1^{\vphantom{n}}_{\vphantom{n}}   & \T^{\vphantom{n}}_{\vphantom{n}}   & \GC^n & {\T_Y}^{\vphantom{n}}_{\vphantom{n}} := \Spec(\C[M_Y]) & 1    \\};
 \path[-stealth]
  	(m-1-1) edge    (m-1-2)
		(m-1-2) edge    (m-1-3)
		(m-1-3) edge    (m-1-4)
		(m-1-4) edge    (m-1-5)
		(m-2-1) edge    (m-2-2)
		(m-2-2) edge    (m-2-3)
		(m-2-3) edge    (m-2-4)
		(m-2-4) edge    (m-2-5)
		(m-1-2) edge node [right] { $ {\tau} $ }   (m-2-2)
		(m-1-3) edge node [right] { $ {\tau'} $ }   (m-2-3)
		(m-1-4) edge node [right] { $ {\tau_Y} $ }   (m-2-4) ;
	\end{tikzpicture}
\end{center} 
There exists an injective morphism $F : N \to N'$ and  a surjective homomorphism $P : N' \to N_Y$, and the following   diagrams of free $\Z$-modules commute:  
\begin{center}
\begin{tikzpicture}
  \matrix (m) [matrix of math nodes,row sep=0.7em,column sep=7em,minimum width=2em]
  {
	0_{\vphantom{Y}}{\vphantom{'}}   & N_{\vphantom{Y}}{\vphantom{'}}   & N'_{\vphantom{Y}}  & N_Y{\vphantom{'}} & 0_{\vphantom{Y}}{\vphantom{'}}     \\
	0_{\vphantom{Y}}{\vphantom{'}}   & N_{\vphantom{Y}}{\vphantom{'}}    & N'_{\vphantom{Y}}  & N_Y{\vphantom{'}} & 0_{\vphantom{Y}}{\vphantom{'}}     \\};
 \path[-stealth]
   	(m-1-1) edge    (m-1-2)
		(m-1-2) edge node [above] {$F$ }   (m-1-3)
		(m-1-3) edge node [above] {$P$ }     (m-1-4)
		(m-1-4) edge    (m-1-5)
		(m-2-1) edge    (m-2-2)
		(m-2-2) edge  node [above] {$F$ }    (m-2-3)
		(m-2-3) edge  node [above] {$P$ }    (m-2-4)
		(m-2-4) edge    (m-2-5)
		(m-1-2) edge node [right] { $  \hat{\tau }$ }   (m-2-2)
		(m-1-3) edge node [right] { $ \hat{\tau}'  $ }   (m-2-3)
		(m-1-4) edge node [right] { $ \hat{\tau_Y}$ }   (m-2-4)
 ;
	\end{tikzpicture}
\end{center}
\begin{center}
\begin{tikzpicture}
  \matrix (m) [matrix of math nodes,row sep=0.7em,column sep=7em,minimum width=2em]
  {
	0_{\vphantom{Y}}{\vphantom{'}}   & M_Y{\vphantom{'}}  & M'_{\vphantom{Y}}  & M_{\vphantom{Y}}   & 0_{\vphantom{Y}}{\vphantom{'}}     \\
	0_{\vphantom{Y}}{\vphantom{'}}   & M_Y{\vphantom{'}}  & M'_{\vphantom{Y}}  & M_{\vphantom{Y}}   & 0_{\vphantom{Y}}{\vphantom{'}}     \\};
 \path[-stealth]
   	(m-1-1) edge    (m-1-2)
		(m-1-2) edge   node [above] {$P^*$ }   (m-1-3)
		(m-1-3) edge   node [above] {$F^*$ }  (m-1-4)
		(m-1-4) edge    (m-1-5)
		(m-2-1) edge    (m-2-2)
		(m-2-2) edge   node [above] {$P^*$ }    (m-2-3)
		(m-2-3) edge   node [above] {$F^*$ }    (m-2-4)
		(m-2-4) edge    (m-2-5)
		(m-1-2) edge node [right] { $ \tilde{\tau_Y}$ }   (m-2-2)
		(m-1-3) edge node [right] { $ \tilde{\tau}'  $ }   (m-2-3)
		(m-1-4) edge node [right] { $ \tilde{\tau}$ }   (m-2-4)
 ;
	\end{tikzpicture}
\end{center}
There always exists a section $s^* : M \to M' $, but not always a $\Gamma$-equivariant one. Therefore, we obtain a section $\T_Y \to \GC^n $, but not always a $\Gamma$-equivariant one. In other words, $\GC^n \cong \T \times \T_Y$, but this isomorphism is not always $\Gamma$-equivariant.
\end{remark}

\begin{example} \label{importantex}
The  real tori  $\GR$ and   $\mathbb{S}^1$ are real subtori of $\RC(\GC)$. The inclusion are given by:
$$(\GC, \tau_0)  \to (\GC^2, \tau_2), \ t   \mapsto (t, t)
\ \text{ and } \ 
(\GC, \tau_1)   \to (\GC^2, \tau_2), \  t   \mapsto (t, t^{-1}).$$
We obtain the diagrams of Remark \ref{FracsectionR} with $M'=\Z^2$, $M=\Z$,  $M_Y=\Z$, $F^* =[1,1]$  and $P^* =[1,-1]$ for $\GR$, and   $F^* =[1,-1]$  and $P^* =[1,1]$ for $\mathbb{S}^1$. In these two cases, there does not exist a $\Gamma$-equivariant section since $\tau_2$ is not equivalent to $\tau_1 \times \tau_0$. 
\end{example}

In the case where  $\RC(\GC)$ is a subtorus of a real torus $(\GC^n, \tau')$, we have the following result: 

\begin{lemma}   \label{SectionGalois}  
Let $(\GC^{2q}, \tau_2^{\times q})$ be  a subtorus of $(\GC^n, \tau')$. Let $M:=\Hom_{gr}(\GC^{2q}, \GC)$ and $M':=\Hom_{gr}(\GC^n, \GC)$. Then, there exists a $\Gamma$-equivariant section $s^* : M \to M'$ (i.e. $F^* \circ s^* =id$ and $\tau' \circ s^* = s^* \circ \tau_2^{\times q}$). 
\end{lemma}

\begin{proof} 
The Galois group $\Gamma =\{id, \gamma \}$ acts on $M$ via   $\tilde{\tau}_2^{\times q}$, so $M$ is a $\Gamma$-module. We have the following short exact sequences of $\Z[\Gamma]$-modules:
$$ 0 \xrightarrow{ \ \ \ } M_Y \xrightarrow{\  P^* \  } M' \xrightarrow{ \  F^* \  } M \xrightarrow{ \ \ \ } 0.$$ 
 Note that we have   an isomorphism of $\Z[\Gamma]$-module:   
$$ M    \to \Z[\Gamma]^q,  \ \  (k_1,l_1, \dots, k_q, l_q)   \mapsto (k_1   \chi^{id} + l_1   \chi^{\gamma}, \dots ,  k_q   \chi^{id} + l_q   \chi^{\gamma}).$$ 
Hence $M$ is free $\Z[\Gamma]$-module of rank $q$, so it is a projective $\Z[\Gamma]$-module. By  \cite[Proposition A3.1]{Eisenbud},  there  exists a  morphism $s^* : M \to M'$ of $\Z[\Gamma]$-module  such that $F^* \circ s^* = id_{M}$. 
\end{proof}

 \begin{remark} \
\begin{enumerate}[leftmargin=0.75cm, label=(\roman*)]
\item The interpretation of Lemma \ref{SectionGalois} is : $(\GC^n, \tau') \cong \RC(\GC)^q \times T' $, where $T'$ is a real torus of dimension $n-2q$. 
\item For $\GR$-actions, the $\Z[\Gamma]$-module $M = \Z$, with $\Gamma$-action given by $\tilde{\tau}_0 = id$, is not a projective $\Z[\Gamma]$-module. Indeed, we have a $\Gamma$-equivariant isomorphism  
$$M    \to \Z[\Gamma]/(\chi^{\gamma}), \ \ m  \mapsto [ m \chi^{id} +   m     \chi^{\gamma} ].$$ 
\item For $\mathbb{S}^1$-actions, the $\Z[\Gamma]$-module $M = \Z$, with $\Gamma$-action given by $\tilde{\tau}_1 = -id$, is not a projective $\Z[\Gamma]$-module. Indeed, we have a $\Gamma$-equivariant isomorphism 
$$M    \to \Z[\Gamma]/(\chi^{\gamma}), \ \  m  \mapsto [ m \chi^{id} + (-m ) \chi^{\gamma} ].$$ 
\end{enumerate}
\end{remark}

\subsection{Real torus actions} \label{SubsectionRTA} 

We  now consider   actions  of   real tori  on   $\R$-varieties.  

\begin{lemma}  \label{defaction}
Let $T $ be a real  torus. There is a one-to-one correspondence between    
 quasi-projective $\R$-varieties    endowed with a $T $-action  and   tuples $(\T, \tau , X, \sigma , \mu)$   consisting of: 
\begin{enumerate}[leftmargin=0.75cm, label=(\roman*)]
\item   a complex torus $\T$ endowed with an $\R$-group structure $\tau$ such that $ \T/\langle \tau \rangle \cong T $;
\item a quasi-projective $\C$-variety $X$ endowed with an $\R$-structure $\sigma$; 
\item    an  action $\mu : \T \times X \to X$   such that the following diagram commutes:
\begin{center}
\begin{tikzpicture}
\matrix (m) [matrix of math nodes,row sep=1em,column sep=7em,minimum width=2em]
 {
          \T \times  X   &    X      \\
          \T \times  X   &    X         \\};
 \path[-stealth]
    (m-1-1) edge node  [above]{ $\mu $} (m-1-2)
            edge node  [left]{${\tau}  \times {\sigma }  $} (m-2-1)
    (m-1-2) edge node  [right]{${\sigma } $}  (m-2-2)
    (m-2-1) edge node  [below]{$\mu$} (m-2-2);
\end{tikzpicture}
\end{center}
\end{enumerate}
\end{lemma}

\begin{proof} 
Let $(\T, \tau , X, \sigma , \mu)$ be such a tuple. By Proposition
\ref{eqcat}, the morphism $\mu : \T \times X \to X$ induces a morphism $ \mu_0 : (\T \times X)/{\langle \tau \times \sigma \rangle} \to X/{\langle   \sigma \rangle}$.  Since $(\T \times X)/{\langle \tau \times \sigma \rangle} \cong  \T/{\langle \tau   \rangle}  \times X/{\langle   \sigma \rangle}$, we have a  $\T/{\langle \tau   \rangle}$-action on  $X/{\langle   \sigma \rangle}$. \\
Conversely, let $X$ be an $\R$-variety endowed with a $T$-action $  T \times X \to X$. Since $(T \times X)_{\C} \cong   T_{\C} \times  X_{\C}$, we obtain an action $\mu := T_{ \C} \times    X_{\C} \to   X_{\C}$ satisfying the commutative diagram:
\begin{center}
\begin{tikzpicture}
\matrix (m) [matrix of math nodes,row sep=1em,column sep=7em,minimum width=2em]
 {
          T_{\C} \times  X_{\C}   &    X_{\C}      \\
          T_{\C} \times  X_{\C}   &    X_{\C}         \\};
 \path[-stealth]
    (m-1-1) edge node  [above]{ \small{$\mu $}} (m-1-2)
            edge node  [left]{ \small{${id \times (z \mapsto \bar{z})} $}}  (m-2-1)
						edge node  [right] {  \small{${id \times (z \mapsto \bar{z}) }  $}} (m-2-1)
    (m-1-2) edge node  [right]{\small{${id \times (z \mapsto \bar{z}) } $}}  (m-2-2)
    (m-2-1) edge node  [below]{\small{$\mu$}} (m-2-2);
\end{tikzpicture}
\end{center}
Which ends the proof. 
\end{proof}

\begin{example} \label{RWA3}
Consider the action of $\GC^2$ on $\AC^3$ given by $(s,t) \cdot (x,y,z)= (sx, ty, stz)$. The Weil restriction $(\GC^2, \tau_2)$ acts on $(\AC^3, \sigma')$, where $\sigma'(x,y,z) = (\overline{y}, \overline{x}, \overline{z})$.  
\end{example}

\begin{example}   \label{RWA4}
Consider  the hypersurface $X$ of $\AC^4 := \Spec(\C[x_1, x_2, x_3, x_4])$ defined by $x_1x_3=x_2x_4$. The  torus $\GC^2$ acts on $\AC^4$ by $(s,t) \cdot (x_1, x_2, x_3, x_4) := (s x_1,t x_2, st^2x_3, s^2t x_4)$. Since the polynomial $x_1x_3-x_2x_4$ is homogeneous, $\GC^2$ acts on $X$. Let $\sigma'$ be the $\R$-structure on $\AC^4$ defined by $\sigma'(x_1, x_2, x_3, x_4) = (\overline{x_2}, \overline{x_1}, \overline{x_4}, \overline{x_3})$ and let $\sigma$ be the induced $\R$-structure on $X$. Then, the real torus $(\GC^2, \tau_2)$ acts on $(\AC^4, \sigma')$ and on $(X, \sigma)$. 
\end{example}

Let us note that if a real torus $(\T, \tau)$  acts on an affine variety $(X, \sigma)$, then the comorphism ${\sigma }^{\sharp}$ of ${\sigma }$ preserves the $M$-grading of the algebra $\C[X]$, where $M:=\Hom_{gr}(\T, \GC)$. This observation will be useful in the proof of Proposition \ref{GaloisToricDowngrading}.

\begin{lemma} \label{IsomGradedStructure}  
Let $(\T, \tau )$ be a real torus  acting   on the affine   $\R$-variety $(X , \sigma )$. Let $M := \Hom_{gr}(\T,\GC)$ and  let $\omega_M$ be the weight cone of the $\T$-action on $X$.  Then $\tilde{\tau}(\omega_M) = \omega_M$ and for all $m \in M$:
\begin{equation*}
{\sigma }^{\sharp} \left(\C[X]_m \right) = \C[X]_{\tilde{\tau}(m)}.
\end{equation*}
\end{lemma}

\begin{proof}    
Let  $m \in M$ and let $f \in \C[X]_m$. We obtain from the diagram of Lemma \ref{defaction}: 
\begin{equation*}
(\mu^{\sharp}  \circ {\sigma }^{\sharp} )(f)  = (({\tau}^{\sharp}  \times {\sigma}^{\sharp} ) \circ \mu^{\sharp})(f) 
=  {\tau}^{\sharp} (\chi^m) \otimes  {\sigma }^{\sharp}  (f) 
=   \chi^{\tilde{\tau}(m)}  \otimes   {\sigma }^{\sharp} (f). 
\end{equation*}
Hence  ${\sigma }^{\sharp} \left(\C[X]_m \right) \subset \C[X]_{\tilde{\tau}(m)} $.  Moreover, if $g \in \C[X]_{\tilde{\tau}(m)}$, then $g = {\sigma }^{\sharp}( {\sigma }^{\sharp}(g))$. Hence, 
${\sigma }^{\sharp} \left(\C[X]_m \right) = \C[X]_{\tilde{\tau}(m)}$.
\end{proof}

\subsection{The case of affine toric $\R$-varieties}  \label{SubsectionATV}

In this subsection, we consider the particular case of      affine toric     $\R$-varieties, i.e.   affine $\R$-varieties $X$ such that $X_{\C}$ is an affine toric $\C$-variety.

\begin{proposition}  \label{ActionToric}
Let $(\T, \tau)$ be a real torus, let     $M := \Hom_{gr}(\T,\GC)$ and $N$ be its dual lattice.  Let    $\delta$  be a pointed cone in $N_{\Q}$ and let  $X_{\delta}$  be the associated affine     toric $\C$-variety.     
The torus $(\T, \tau)$ acts on the affine toric  $\R$-variety $(X_{\delta}, \sigma)$, where $\sigma$ is an $\R$-structure on $X_{\delta}$,   if and only if  there exists  an $\R$-group structure        $\tau'$ on $\T$ equivalent to $\tau$ such that $\hat{\tau}'(\delta) = \delta$.
\end{proposition}
 
\begin{proof} (Compare with \cite[Proposition 1.19]{Huruguen}). 
Assume that $\tau$ is equivalent to $\tau'$ and $\hat{\tau}'(\delta)=\delta$. Recall that we denote     $\C[M] := \{ \sum_{m \in M} a_m  \chi^m, a_m \in \C \}$ and    $\C[S_{\delta}] := \{ \sum_{m \in  \delta^{\vee} \cap M} a_m \chi^m, \\ 
a_m \in \C   \} $ the coordinate rings of   $\GC^n$ and $X_{\delta}$ respectively. Since $\tilde{\tau}'(\delta^{\vee})=\delta^{\vee}$, the algebra automorphism  
\begin{equation*}
\tau^{\sharp} : \C[M]  \to \C[M], \ \
\sum_{m \in M} a_m \chi^m    \mapsto \sum_{m \in M} \overline{a}_m \chi^{\tilde{\tau}'(m)} 
\end{equation*}
can be restricted to $\C[S_{\delta}] \subset \C[M]$.  It is the comorphism of an $\R$-structure $\sigma$ on $X_{\delta}$.  

Conversely, let $\sigma$ be an $\R$-structure on $X_{\delta}$ such that  $(\T, \tau)$ acts on   $(X_{\delta}, \sigma)$.  Let $m \in \delta^{\vee} \cap M$, then there exists $\tilde{\varphi} \in \GL(M)$ such that $\chi^m \in \C[S_{\delta}]_{\tilde{\varphi}(m)}$. Let $\varphi \in \Aut_{gr}(\T)$ be the corresponding automorphism.   By  Lemma \ref{IsomGradedStructure}, $\sigma^{\sharp}(\chi^m) \in \C[S_{\delta}]_{\tilde{\tau}(\tilde{\varphi}(m))}$, so  $\sigma^{\sharp}(\chi^m)= \chi^{(\tilde{\varphi}^{-1} \circ \tilde{\tau} \circ \tilde{\varphi})(m)}$. Let $\tau':= \varphi^{-1} \circ \tau \circ  \varphi$ be an $\R$-group structure on $\T$ equivalent to $\tau$, then $\sigma^{\sharp}(\chi^m) = \chi^{\tilde{\tau}'(m)}$. Hence, ${\tilde{\tau}'(m)} \in \delta^{\vee} \cap M$ and  $\hat{\tau}'(\delta) = \delta$. 
\end{proof}

\begin{remark} The weight cone of the $\T$-action on $X_{\delta}$ does not always coincide with $\delta^{\vee}$ (compare with Lemma \ref{IsomGradedStructure}).
\end{remark}

\begin{remark} The Weil restriction $\RC(\GC)$ acts on a 2-dimensional affine   toric  $\R$-variety  $(X_{\delta}, \sigma)$ if and only if there exists a basis $\{e_1, e_2\}$ of the lattice $N$ such that the cone $\delta$ is symmetric with respect to the line $\Q(e_1+e_2)$. Let $\sigma$ be the $\R$-structure on $\AC^2$ defined by  $\sigma(x,y)=(\overline{y}, \overline{x})$. The    toric $\R$-variety $(\AC^2, \sigma)$     is endowed with an $\RC(\GC)$-action    since $\hat{\tau}_2(\Q^2_{\geq 0} ) = \Q^2_{\geq 0}$.  
\end{remark}

\begin{example} \label{RW3t}
Consider the cone $\delta= \Q^3_{\geq 0 }$ in $N_{\Q} = \Q^3$ spanned by the canonical basis of $N$. Let $X_{\delta}:= \AC^3$ be the associated  toric $\C$-variety. The cone $\delta$ is stable under the lattice involution $\hat{\tau}':= \hat{\tau}_2 \times \hat{\tau}_0$ induced by the $\R$-group structure on $\GC^3$ defined by $\tau':= \tau_2 \times \tau_0$. Consider the   $\R$-structure $\sigma'$ on $\AC^3$  defined by $\sigma'(x,y,z)=(\overline{y}, \overline{x}, \overline{z})$.  The natural action of $\GC^3$ on $\AC^3$ is compatible with the $\R$-structures $\tau'$ and $\sigma'$, i.e. $R_{\C/\R}(\GC) \times \GR$ acts on $(\AC^3, \sigma')$. Note that the action of  $R_{\C/\R}(\GC)$ on $(\AC^3, \sigma')$ given in  Example \ref{RWA3} comes from  the action of $R_{\C/\R}(\GC) \times \GR$ on $(\AC^3, \sigma')$ (details in Example \ref{RWA3''}).   
\end{example}

\begin{counterexample} Consider the cone $\delta = \text{Cone} \left\{ 
\begin{bmatrix} 1 \\ 0   \end{bmatrix} ; \begin{bmatrix} 1 \\ 2  \end{bmatrix} \right\}$ in $N_{\Q} = \Q^2$.  There are no $\R$-group structure $\tau$ equivalent to $\tau_2$ such that $\hat{\tau}(\delta)=\delta$, so we cannot endow   $X_{\delta}$ with an $\R$-structure compatible with a $R_{\C/\R}(\GC)$-action.
\end{counterexample}

Let $(\T, \tau)$ be a real torus and let  $\sigma$ be an $\R$-structure on an $n$-dimensional toric $\C$-variety $X_{\delta}$ induced by an $\R$-group structure $\tau'$ on $\GC^n$. By a  $(\T, \tau)$-action   on  $(X_{\delta}, \sigma)$, we mean a $(\T, \tau)$-action  such that $(\T, \tau)$ is a real subtorus of  $(\GC^n, \tau')$. Let's now have a  look at   $(\T, \tau)$-actions   on $(X_{\delta}, \sigma)$.

\begin{corollary} \label{ToricRW}
Let     $M := \Hom_{gr}(\GC^n,\GC)$ and $N$ be its dual lattice.  Let  $\delta$ be a pointed cone in $N_{\Q}$ and let  $X_{\delta}$  be the associated affine     toric $\C$-variety.
The torus $(\GC^{2q}, \tau_2^{\times q})$ acts on the affine toric  $\R$-variety $(X_{\delta}, \sigma)$, where $\sigma$ is an $\R$-structure on $X_{\delta}$,   if and only if  there exists  an $\R$-group structure       $\tau'$ on $\GC^n$ equivalent to an $\R$-group structure of the form $\tau_2^{\times q} \times \tau''$ and  such that $\hat{\tau}'(\delta) =\delta$, where $\tau''$ is an $\R$-group structure on $\GC^{n-2q}$.     
\end{corollary}

\begin{proof}
It is a consequence of Lemma \ref{SectionGalois}  and Proposition \ref{ActionToric}.
\end{proof}

The Corollary \ref{ToricRW} is specific to   Weil restriction actions:

\begin{example} \label{exGMRa} 
 Let $N=\Z^2$,  let $\delta= \Q^2_{\geq 0}$ be a pointed cone in $N_{\Q}$, and let $\AC^2$ be the associated affine toric $\C$-variety. Since $\hat{\tau}_2(\delta) = \delta$, the $\R$-group structure $\tau_2$ on $\GC^2$ extends to an $\R$-structure $\sigma$ on $\AC^2$ defined by $\sigma(x,y) = (\overline{y}, \overline{x})$. Note that the real torus $(\GC, \tau_0)$ acts on $(\AC^2, \sigma)$ by $ t \cdot (x,y)   = (tx, ty)$, but $\tau_2$ is not equivalent to $\tau_0 \times \tau_1$  (see Example \ref{importantex}). 
\end{example}

\section{Altmann-Hausen presentation for normal affine $\C$-varieties}   \label{SectionCAHP}
 
In this  section, we introduce  the group of tailed polyhedra, which will serve  as the group of coefficients for the polyhedral divisors, and we recall  the main results  obtained by Altmann-Hausen in \cite{Alt}. We also   recall some basic facts   about convex geometry. Our main references for this are Altmann-Hausen article's \cite{Alt} and Fulton book's \cite{Fult}.

\subsection{Tailed polyhedra and polyhedral divisors} \label{SubsectionTPPD}

A subset $\Pi \subset N_{\Q}$ is called a \textit{polytope} if there exists a finite set  $S \subset N_{\Q}$ such that $\Pi$ is the convex hull of $S$,  and  
it  is called a    \textit{rational polytope} if $S$ can be taken inside the lattice $  N$. 
A proper face $\Pi'$ of $\Pi$ is 
the intersection of $\Pi$ with a supporting affine hyperplane.

A \textit{convex polyhedron} is the intersection of finitely many closed affine half spaces in $N_{\Q}$.  
For us, a \textit{polyhedron} in $N_{\Q}$ is always a convex polyhedron.
The \textit{relative interior} of a  polyhedron $\Delta$, denoted by  $\text{Relint}(\Delta )$, is obtained  by removing all  {proper}  {faces} from $\Delta$. 
Moreover, any polyhedron $\Delta$ in $ N_{\Q}$ admits a {Minkowski}  {sum} decomposition: 
\begin{equation*}
\Delta = \Pi + \omega_N,
\end{equation*}
where $\Pi \subset N_{\Q}$ is a polytope and $\omega_N \subset N_{\Q}$ is a cone. In this decomposition, the cone $\omega_N$ is unique and called the \textit{tail cone} of $\Delta$ (see \cite[\S 1]{Alt}).

\begin{example}  $\Delta = \Pi + \omega_N$  
\begin{multicols}{5}
\small{
\begin{flushleft}
\definecolor{ffqqqq}{rgb}{1.,0.,0.}
\begin{tikzpicture}[line cap=round,line join=round,>=triangle 45,x=1.0cm,y=1.0cm]
\begin{axis}[
x=1.0cm,y=1.0cm,
axis lines=middle,
ymajorgrids=false,
xmajorgrids=false,
xmin=-0.05,
xmax=2.25,
ymin=-0.20,
ymax=1.25,
xtick={-3.0,-2.0,...,3.0},
ytick={-3.0,-2.0,...,3.0},
]
\clip(-3.,-3.) rectangle (3.,3.);
\shade[top color = white, bottom color = red] (0.,0.5) -- (0.,2.) -- (3.,2.) --  (1.,0.)  -- cycle;
\draw [line width=2.pt,color=ffqqqq] (0.,0.5)-- (0.,2.);
\draw [line width=2.pt,color=ffqqqq] (0.,2.)-- (3.,2.);
\draw [line width=2.pt,color=ffqqqq] (3.,2.)-- (1.,0.) ;
\draw [line width=2.pt,color=ffqqqq] (1.,0.)-- (0.,0.5);
\node[label=right:{ \textbf{$\Delta$}   }] (n1) at (0.5,0.75) {};
\end{axis}
\end{tikzpicture}
\end{flushleft}

\begin{center}
$\begin{matrix}
\ \\
= \\
\ \\
\end{matrix}$
\end{center}

\begin{center}
\definecolor{ffqqqq}{rgb}{1.,0.,0.}
\begin{tikzpicture}[line cap=round,line join=round,>=triangle 45,x=1.0cm,y=1.0cm]
\begin{axis}[
x=1.0cm,y=1.0cm,
axis lines=middle,
ymajorgrids=false,
xmajorgrids=false,
xmin=-0.20,
xmax=1.5,
ymin=-0.20,
ymax=1.35,
xtick={-3.0,-2.0,...,3.0},
ytick={-3.0,-2.0,...,3.0},
]
\clip(-3.,-3.) rectangle (3.,3.);
\fill[line width=2.pt,color=ffqqqq,fill=ffqqqq,fill opacity=1.0] (0.,0.5) -- (1.,0.)  -- cycle;
\draw [line width=2.pt,color=ffqqqq] (0.,0.5)-- (1.,0.)  node [above] { \ \  \textcolor[rgb]{0,0,0}{\textbf{$\Pi$}} \ \    };
\end{axis}
\end{tikzpicture}
\end{center}

\begin{center}
$\begin{matrix}
\ \\
+ \\
\ \\
\end{matrix}$
\end{center}

\begin{center}
\definecolor{ffqqqq}{rgb}{1.,0.,0.}
\begin{tikzpicture}[line cap=round,line join=round,>=triangle 45,x=1.0cm,y=1.0cm]
\begin{axis}[
x=1.0cm,y=1.0cm,
axis lines=middle,
ymajorgrids=false,
xmajorgrids=false,
xmin=-0.20,
xmax=1.5,
ymin=-0.20,
ymax=1.35,
xtick={-3.0,-2.0,...,3.0},
ytick={-3.0,-2.0,...,3.0},
]
\clip(-3.,-3.) rectangle (3.,3.);
\shade[top color = white, bottom color = red]  (0.,0.) -- (0.,2.) -- (3.,3.) --  cycle;
\draw [line width=2.pt,color=ffqqqq] (0.,0)-- (0.,2.) ;
\draw [line width=2.pt,color=ffqqqq] (0.,2.)-- (3.,3.);
\draw [line width=2.pt,color=ffqqqq] (3.,3.)-- (0.,0.);
\node[label=right:{ \textbf{$\omega_N$} }] (n1) at (0.15,1.) {};
\end{axis}
\end{tikzpicture}
\end{center}}
\end{multicols}
\end{example}

\begin{definition}   
Let $\omega_N$ be a pointed cone in $N_{\Q}$.
By a $\omega_N$-\textit{polyhedron} in $N_{\Q}$, we mean a polyhedron   in $N_{\Q}$ having the cone $\omega_N$ as its tail cone. We denote the set of all $\omega_N$-polyhedra in $ N_{\Q}$ by $\text{Pol}_{\omega_N}^+(N_{\Q})$.
\end{definition}

The Minkowski sum of two  {$\omega_N$-polyhedra} in $ N_{\Q}$ is again a $\omega_N$-polyhedron in $ N_{\Q}$. Thus, endowed with Minkowski sum, $\text{Pol}_{\omega_N}^+(N_{\Q})$ is an abelian monoid, whose neutral element is $\omega_N$ \cite[\S 1]{Alt}.

We now introduce the language of polyhedral divisors and proper polyhedral divisors.  The idea is to replace   rational   coefficient  by tailed polyhedra    \cite[\S 2]{Alt}.

Let $Y$ be a normal $\C$-variety. The group of Weil divisors on $Y$ is denoted $\text{WDiv}(Y)$ and the group of Cartier divisors on $Y$ is denoted by $\text{CDiv}(Y)$. Since $Y$ is normal, we have an inclusion $\text{CDiv}(Y) \subset \text{WDiv}(Y)$. A Cartier (resp. Weil) $\Q$-divisor is an element of $\Q \otimes_{\Z} \text{CDiv}(Y)$ (resp $\Q \otimes_{\Z} \text{WDiv}(Y)$). The sheaf of sections $\mathcal{O}(D)$ of a Weil $\Q$-divisor $D$ on  $Y$ is defined by:
\begin{equation*}
\H^0(V, \mathcal{O}(D)) := \{ f \in \C(Y) \ | \ \Div_V(f \big{|}_{V}) + D \big{|}_{V}   \geq 0 \} \cup \{ 0 \},
\end{equation*}
where $V \subset Y$ is an open subset.
Now we turn to divisors with tailed polyhedra coefficients. Let $\omega_N$ be a pointed cone in $N_{\Q}$.
An $\omega_N$-\textit{polyhedral} \textit{divisor} on $Y$ is a formal sum:
\begin{equation*}
 \D = \sum_{Z} \Delta_Z \otimes Z \in \text{Pol}_{\omega_N}^+(N_{\Q}) \otimes_{\Z} \text{WDiv}(Y)
\end{equation*}
over all prime divisors $Z \subset Y$,  and $\Delta_Z=\omega_N$ for all but finitely prime divisors $Z$.

Let $\D = \sum_{Z} \Delta_Z \otimes Z$ be a $\omega_N$-polyhedral divisor on $Y$. For a prime divisor $Z$ on $Y$ we denote the support function of $\Delta_Z$ by 
\begin{equation*}
h_{Z} : \omega_N^{\vee} \to \Q, \ m \mapsto  \text{min}\{ \langle m,v \rangle \ | \ v \in \Delta_Z \}.
\end{equation*}
For every $m \in \omega_N^{\vee}$ we can evaluate $\D$ in $m$ by  letting $\D(m)$ be the Weil $\Q$-divisor on $Y$ defined by:
\begin{equation*}
\D(m) := \sum_{Z} h_Z(m) \otimes Z \in \Q \otimes_{\Z} \text{WDiv}(Y).
\end{equation*}

Before introducing   proper polyhedral divisors, we  recall the following definitions:

\begin{definition} A Cartier $\Q$-divisor $D$ on $Y$ is called \textit{semi-ample} if, for some $n \in \N^*$, the set of open subsets $Y_f :=Y \backslash \text{Supp}(\Div(f) + D)$, with  $f \in \H^0(Y, \mathcal{O}_Y(nD)$, cover $Y$. A Cartier $\Q$-divisor $D $ on $Y$ is called  \textit{big} if, for some $n \in \N^*$, there exists a section $f \in \H^0(Y, \mathcal{O}_Y(nD)$ with an affine non-vanishing locus $Y_f$. 
\end{definition}

\begin{definition} \label{ppDiv}
A \textit{proper} $\omega_N$-\textit{polyhedral} \textit{divisor} on $Y$, abbreviated an $\omega_N$-pp-divisor, is an $\omega_N$-polyhedral divisor $\D = \sum_{Z} \Delta_Z \otimes Z $ on $Y$ satisfying the following properties:
\begin{enumerate}[leftmargin=0.75cm, label=(\roman*)]
\item  for all $m \in \omega_N^{\vee} \cap M, \ \D(m)$ is a semi-ample Cartier $\Q$-divisor on $Y$; and
\item for all $m \in \text{Relint}(\omega_N^{\vee}) \cap M, \ \D(m)$ is big.
\end{enumerate}
\end{definition}

The sum of two $\omega_N$-pp-divisors with respect to a given  cone $\omega_N$ is again an $\omega_N$-pp-divisor. Thus, $\omega_N$-pp-divisors form a monoid denoted by $\PPDiv(Y, \omega_N)$.

\begin{example} \label{Eval} Let $N = \Z^2$, let $\omega_N = \Q_{\geq 0 }^2 $, and let $\Delta$  be the $\omega_N$-polyhedron defined below. The normal quasifan  associated to $\Delta$ consists of the two cones $\delta_1$ and $\delta_2$ refining the  cone $\omega_N^{\vee} = \Q_{\geq 0 }^2$ of the dual lattice $M = \Z^2$ (see \cite[\S 1.1.2]{LienT} for details). The support function $h_{\Delta}  : \omega_N^{\vee} \to \Q, \ m \mapsto  \text{min}\{ \langle m,v \rangle \ | \ v \in \Delta  \} $ is linear on each $\delta_i$, and we obtain:
\begin{multicols}{3}
\ \vspace{-1.15cm}
\begin{center}
\begin{equation*}
h_{\Delta}(m_1, m_2) = \left\{
														\begin{array}{ll}
														m_2  & \mbox{if } (m_1, m_2) \in \delta_1 \\
														m_1   & \mbox{if } (m_1, m_2) \in \delta_2\\
														\end{array}
												\right.	
\end{equation*}
\ \vspace{2cm}
\end{center}

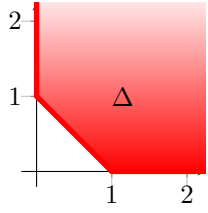
\begin{flushright}
\definecolor{ffqqqq}{rgb}{1.,0.,0.}
\begin{tikzpicture}[line cap=round,line join=round,>=triangle 45,x=1.0cm,y=1.0cm]
\begin{axis}[
x=1.0cm,y=1.0cm,
axis lines=middle,
ymajorgrids=true,
xmajorgrids=true,
xmin=-0.2,
xmax=2.25,
ymin=-0.2,
ymax=2.25,
xtick={-3.0,-2.0,...,3.0},
ytick={-3.0,-2.0,...,3.0},]
\clip(-3.,-3.) rectangle (3.,3.);
\shade[top color = white, bottom color = red]  (0.,1) -- (0.,2.5) -- (2.5,2.5) -- (2.5,0.) -- (1.,0.)  -- cycle;
\draw [line width=2.pt,color=ffqqqq] (0.,1)-- (0.,2.5);
\draw [line width=2.pt,color=ffqqqq] (0.,2.5)-- (2.5,2.5);
\draw [line width=2.pt,color=ffqqqq] (2.5,2.5)-- (2.5,0.);
\draw [line width=2.pt,color=ffqqqq] (2.5,0.)-- (1.,0.);
\draw [line width=2.pt,color=ffqqqq] (1.,0.)-- (0.,1);
\node[label=right:{$ {\Delta } $}] (n1) at (0.75,1.) {};
\end{axis}
\end{tikzpicture}
\end{flushright}

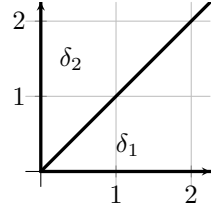
\begin{flushright}
\begin{tikzpicture}[line cap=round,line join=round,>=triangle 45,x=1.0cm,y=1.0cm]
\begin{axis}[
x=1.0cm,y=1.0cm,
axis lines=middle,
ymajorgrids=true,
xmajorgrids=true,
xmin=-0.2,
xmax=2.25,
ymin=-0.2,
ymax=2.25,
xtick={-3.0,-2.0,...,3.0},
ytick={-3.0,-2.0,...,3.0},]
\clip(-3.,-3.) rectangle (3.,3.);
\draw [-,line width=1.25pt] (0.,0.) -- (0.,3.);
\draw [-,line width=1.25pt] (0.,0.) -- (3.,0.);
\draw [-,line width=1.25pt] (0.,0.) -- (3.,3.);
\node[label=right:{${\delta_1} $}] (n1) at (0.75,0.35) {};
\node[label=right:{${\delta_2} $}] (n1) at (0.,1.5)  {};
\end{axis}
\end{tikzpicture}
\end{flushright}
\end{multicols}
\hspace{-0.85cm}Consider a divisor $\D := \Delta \otimes D $ on a normal  variety $Y$, where $D$ is a prime divisor. Then, 
\begin{equation*}
\D(m_1, m_2)  = \left\{
																\begin{array}{ll}
																m_2  \otimes D & \mbox{if } (m_1, m_2) \in \delta_1 \\
																m_1  \otimes D  & \mbox{if } (m_1, m_2) \in \delta_2\\
																\end{array}
																\right. \\	
\end{equation*}
\end{example}

\begin{example} \label{RWA4'} Let $N = \Z^2$, let $\omega_N = \Q_{\geq 0 }^2 $, let $\Delta_1=\Delta_2=\omega_N$ and let $\Delta_3$ and  $\Delta_4$  be the $\omega_N$-polyhedra defined in the following illustrations. The normal quasifan  associated to   $\Delta_3$ (resp. $\Delta_4$) consists of two cones  refining the  cone $\omega_N^{\vee} = \Q_{\geq 0 }^2$ of the dual lattice $M = \Z^2$. The support function of the polyhedron $\Delta_i$ is denoted $h_{i}  : \omega_N^{\vee} \to \Q, \ m \mapsto  \text{min}\{ \langle m  |  v \rangle \ | \ v \in \Delta_i  \} $. Note that $h_1=h_2=0$.

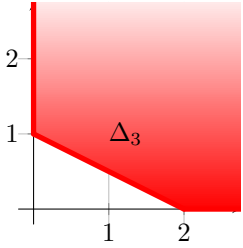
\begin{multicols}{4}
\definecolor{ffqqqq}{rgb}{1.,0.,0.}
\begin{flushleft}
\begin{tikzpicture}[line cap=round,line join=round,>=triangle 45,x=1.0cm,y=1.0cm]
\begin{axis}[
x=1.0cm,y=1.0cm,
axis lines=middle,
ymajorgrids=true,
xmajorgrids=true,
xmin=-0.2,
xmax=2.75,
ymin=-0.2,
ymax=2.75,
xtick={-3.0,-2.0,...,3.0},
ytick={-3.0,-2.0,...,3.0},]
\clip(-3.,-3.) rectangle (3.,3.);
\shade[top color = white, bottom color = red](0.,1.) -- (0.,3.) -- (3.,3.) -- (3.,0.) -- (2.,0.)  -- cycle;
\draw [line width=2.pt,color=ffqqqq] (0.,1.)-- (0.,3.);
\draw [line width=2.pt,color=ffqqqq] (0.,3.)-- (3.,3.);
\draw [line width=2.pt,color=ffqqqq] (3.,3.)-- (3.,0.);
\draw [line width=2.pt,color=ffqqqq] (3.,0.)-- (2.,0.);
\draw [line width=2.pt,color=ffqqqq] (2.,0.)-- (0.,1.);
\node[label=right:{$ {\Delta_3 } $}] (n1) at (0.75,1.) {};
\end{axis}
\end{tikzpicture}
\end{flushleft}

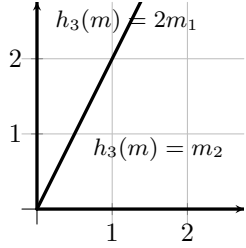
\begin{flushleft}
\begin{tikzpicture}[line cap=round,line join=round,>=triangle 45,x=1.0cm,y=1.0cm]
\begin{axis}[
x=1.0cm,y=1.0cm,
axis lines=middle,
ymajorgrids=true,
xmajorgrids=true,
xmin=-0.2,
xmax=2.75,
ymin=-0.2,
ymax=2.75,
xtick={-3.0,-2.0,...,3.0},
ytick={-3.0,-2.0,...,3.0},]
\clip(-3.,-3.) rectangle (3.,3.);
\draw [-,line width=1.25pt] (0.,0.) -- (0.,3.);
\draw [-,line width=1.25pt] (0.,0.) -- (3.,0.);
\draw [-,line width=1.25pt] (0.,0.) -- (1.5,3.);
\node[label=right:{\small{${h_3(m)=m_2}$}}] (n1) at (0.5,0.8) {};
\node[label=right:{\small{${h_3(m) = 2m_1} $}}] (n1) at (0,2.5)  {};
\end{axis}
\end{tikzpicture}
\end{flushleft}

\begin{flushleft}
\definecolor{ffqqqq}{rgb}{1.,0.,0.}
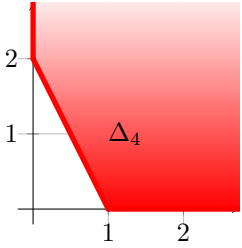
\begin{tikzpicture}[line cap=round,line join=round,>=triangle 45,x=1.0cm,y=1.0cm]
\begin{axis}[
x=1.0cm,y=1.0cm,
axis lines=middle,
ymajorgrids=true,
xmajorgrids=true,
xmin=-0.2,
xmax=2.75,
ymin=-0.2,
ymax=2.75,
xtick={-3.0,-2.0,...,3.0},
ytick={-3.0,-2.0,...,3.0},]
\clip(-3.,-3.) rectangle (3.,3.);
\shade[top color = white, bottom color = red] (0.,2.) -- (0.,3.) -- (3.,3.) -- (3.,0.) -- (1.,0.)  -- cycle;
\draw [line width=2.pt,color=ffqqqq] (0.,2.)-- (0.,3.);
\draw [line width=2.pt,color=ffqqqq] (0.,3.)-- (3.,3.);
\draw [line width=2.pt,color=ffqqqq] (3.,3.)-- (3.,0.);
\draw [line width=2.pt,color=ffqqqq] (3.,0.)-- (1.,0.);
\draw [line width=2.pt,color=ffqqqq] (1.,0.)-- (0.,2.);
\node[label=right:{$ {\Delta_4 } $}] (n1) at (0.75,1.) {};
\end{axis}
\end{tikzpicture}
\end{flushleft}

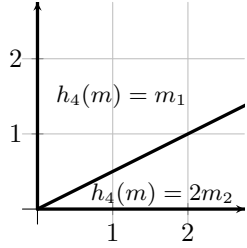
\begin{flushleft}
\begin{tikzpicture}[line cap=round,line join=round,>=triangle 45,x=1.0cm,y=1.0cm]
\begin{axis}[
x=1.0cm,y=1.0cm,
axis lines=middle,
ymajorgrids=true,
xmajorgrids=true,
xmin=-0.2,
xmax=2.75,
ymin=-0.2,
ymax=2.75,
xtick={-3.0,-2.0,...,3.0},
ytick={-3.0,-2.0,...,3.0},]
\clip(-3.,-3.) rectangle (3.,3.);
\draw [-,line width=1.25pt] (0.,0.) -- (0.,3.);
\draw [-,line width=1.25pt] (0.,0.) -- (3.,0.);
\draw [-,line width=1.25pt] (0.,0.) -- (3.,1.5);
\node[label=right:{\small{${h_4(m)=2m_2}$}}] (n1) at (0.46,0.2) {};
\node[label=right:{\small{${h_4(m) = m_1} $}}] (n1) at (0.,1.5)  {};
\end{axis}
\end{tikzpicture}
\end{flushleft} 
\end{multicols}
\hspace{-0.85cm}
Consider the divisor $\D := \Delta_1 \otimes D_1 + \Delta_2 \otimes D_2 +\Delta_3 \otimes D_3 +  \Delta_4 \otimes D_4$ on a normal  variety $Y$, where the $D_i$ are prime divisors. We have $\D  =  \Delta_3 \otimes D_3 +  \Delta_4 \otimes D_4$.  Considering the fan refining these two normal fan, we obtain:
\begin{multicols}{2}
\begin{center}
{\begin{equation*}
\D (m_1, m_2)  = \left\{
																\begin{array}{ll}
																m_2 \otimes D_3 + 2m_2 \otimes D_4 & \mbox{if } (m_1, m_2) \in \delta_1 \\
																m_2  \otimes D_3 + m_1 \otimes D_4 & \mbox{if } (m_1, m_2) \in \delta_2 \\
																2m_1 \otimes D_3 + m_1 \otimes D_4 & \mbox{if } (m_1, m_2) \in \delta_3\\
																\end{array}
																\right. 	
\end{equation*}
\ \\}
\end{center}

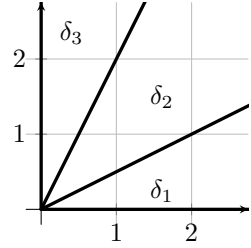
\begin{center}
\ \\ 
\begin{tikzpicture}[line cap=round,line join=round,>=triangle 45,x=1.0cm,y=1.0cm]
\begin{axis}[
x=1.0cm,y=1.0cm,
axis lines=middle,
ymajorgrids=true,
xmajorgrids=true,
xmin=-0.2,
xmax=2.75,
ymin=-0.2,
ymax=2.75,
xtick={-3.0,-2.0,...,3.0},
ytick={-3.0,-2.0,...,3.0},]
\clip(-3.,-3.) rectangle (3.,3.);
\draw [-,line width=1.25pt] (0.,0.) -- (0.,3.);
\draw [-,line width=1.25pt] (0.,0.) -- (3.,0.);
\draw [-,line width=1.25pt] (0.,0.) -- (1.5,3.);
\draw [-,line width=1.25pt] (0.,0.) -- (3.,1.5);
\node[label=right:{$\delta_1$}] (n1) at (1.2,0.25) {};
\node[label=right:{$\delta_2$}] (n1) at (1.2,1.5)  {};
\node[label=right:{$\delta_3$}] (n1) at (0.,2.35)  {};
\end{axis}
\end{tikzpicture}
\end{center}
\end{multicols}
\end{example}

\subsection{Altmann-Hausen presentation} \label{SubsectionCAH}

Let us present the main results of \cite{Alt} about the geometrico-combinatorial presentation of normal affine $\C$-varieties  endowed with a torus action. 

\begin{definition} \label{semiproj} 
A  $\C$-variety $Y$ is said to be \textit{semi-projective} if its $\C$-algebra of global functions $ \H^0(Y, \mathcal{O})$ is finitely generated and $Y$ is projective over $Y_0=\Spec(\H^0(Y, \mathcal{O}))$.
\end{definition}

\begin{remark}
Note that affine varieties and projective varieties are semi-projective. A semi-projective variety is quasi-projective. Indeed,  $Y \to Y_0$ is a projective morphism, moreover $Y_0$ is an affine variety (so quasi-projective).   Then the morphism $Y \to \Spec(\C)$ is quasi-projective. 
\end{remark}

Let $X$  be a $\C$-variety   endowed with an    action of the torus $\T=\Spec(\C[M])$ of  \textit{weight} \textit{cone} $\omega_M \subset M_{\Q}$. We   write  $\C[X] = \bigoplus_{m \in \omega_M \cap M} \C[X]_m$. For all $m \in \omega_M  \cap M  $, we denote:   
\begin{equation*}
\C(X)_m := \left\{ \dfrac{f}{g} \ \Big\lvert  \ \exists k \in M, \ f \in \C[X]_{m+k}, \ \ g \in \C[X]_k \right\} \subset \C(X).
\end{equation*}

Let $Y$ be a normal semi-projective variety, let $\omega_N$ be a pointed cone in $ N_{\Q}$ and let $\D = \sum_{Z} \Delta_Z \otimes Z$ be an $\omega_N$-pp-divisor on $Y$.  By  \cite[proposition 2.11]{Alt},  for all $m, \ m' \ \in M \cap \omega_N^{\vee}$, we have
$\D(m+m') \geq \D(m) + \D(m')$.  
 So,   for all $m, \ m' \ \in \omega_N^{\vee} \cap M$, we have a map:  
\begin{equation*}
\H^0 \left( Y, \mathcal{O}_Y( \D(m)) \right) \otimes \H^0 \left( Y, \mathcal{O}_Y( \D(m')) \right) \to \H^0 \left( Y, \mathcal{O}_Y( \D(m+m')) \right).
\end{equation*}
This ensures that the $\H^0 \left( Y, \mathcal{O}_Y \right)$-sub-modules $\H^0 \left( Y, \mathcal{O}_Y( \D(m)) \right)$ of  $\C(Y)$ can be put together into an $M$-graded $\C$-algebra:
\begin{equation*}
A[Y,\D] := \bigoplus_{m \in \omega_N^{\vee} \cap M} \H^0 \left( Y, \mathcal{O}_Y( \D(m)) \right) \mathfrak{X}_m \subset \C(Y) [M],
\end{equation*}
where $\mathfrak{X}_m $ is an indeterminate of weight $m$. We denote by $X{[Y, \D]} := \Spec(A[Y, \D])$ the associated   $ \T$-scheme. The general idea of the construction of Altmann-Hausen   is to identify $\C(X)_0$ with $\C(Y)$ for some semi-projective variety $Y$, and use an appropriate pp-divisor on $Y$ to construct the grading of $\C[X]$ via an identification between $ \C[X]_m $ and $ \H^0 \left( Y, \mathcal{O}_Y( \D(m)) \right) \mathfrak{X}_m$.

\begin{theorem} \cite[Theorems 3.1 and 3.4]{Alt}. \label{thmAlt}  
Fix a torus $\T$. Let $M$ its character lattice. 
\begin{enumerate}[leftmargin=0.75cm, label=(\roman*)]
\item Let $Y$ be a normal semi-projective variety, let $\omega_N$ be a pointed cone in $ N_{\Q}$, and let $\D$ be a $\omega_N$-pp-divisor on $Y$. 
The affine scheme $X{[Y, \D]} $  is a normal variety, of dimension $\text{dim}(Y) + \text{dim}(\T)$, endowed with a   $\T$-action of weight cone $\omega_N^{\vee}$.
\item Conversely, let $X$ be an affine normal variety endowed with a  $\T$-action, and  let $\omega_N$ be the cone in $N_{\Q}$ dual to the weight cone. There exists a normal semi-projective variety $Y$ and a $\omega_N$-pp-divisor $\D$ on $Y$ such that the graded $\C$-algebras $\C[X]$ and $A[Y, \D]$ are isomorphic. 
\end{enumerate}
\end{theorem}

\begin{example} \label{RWA3'} Consider the Example \ref{RWA3}. The affine variety $\AC^3$ endowed with the action of $\GC^2$ given by $(s,t)\cdot (x,y,z)= (sx, ty, stz)$ is described by  
a semi-projective variety $Y:= \mathbb{P}_{\C}^1 = \AC^1 \cup \{ \infty \}$, and
a pp-divisor on $Y$ defined by $\D := \Delta \otimes \{ \infty \}$, 
 where $\Delta$ is the    polyhedral defined below. Using Example \ref{Eval}, we have:  
\begin{multicols}{2} 

\begin{center}
\ \vspace{-0.5cm}
\begin{equation*}
{\D}(m_1, m_2)  = \left\{
																\begin{array}{ll}
																m_2 \otimes \{ \infty \} & \mbox{if } (m_1, m_2) \in \delta_1 \\
																m_1 \otimes \{ \infty \} & \mbox{if } (m_1, m_2) \in \delta_2\\
																\end{array}
																\right. 	
\end{equation*}
\ \vspace{0.5cm}
\end{center}

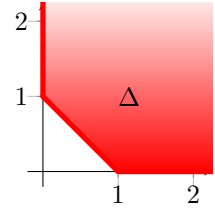
\begin{center}
\definecolor{ffqqqq}{rgb}{1.,0.,0.}
\begin{tikzpicture}[line cap=round,line join=round,>=triangle 45,x=1.0cm,y=1.0cm]
\begin{axis}[
x=1.0cm,y=1.0cm,
axis lines=middle,
ymajorgrids=true,
xmajorgrids=true,
xmin=-0.2,
xmax=2.25,
ymin=-0.2,
ymax=2.25,
xtick={-3.0,-2.0,...,3.0},
ytick={-3.0,-2.0,...,3.0},]
\clip(-3.,-3.) rectangle (3.,3.);
\shade[top color = white, bottom color = red]  (0.,1) -- (0.,2.5) -- (2.5,2.5) -- (2.5,0.) -- (1.,0.)  -- cycle;
\draw [line width=2.pt,color=ffqqqq] (0.,1)-- (0.,2.5);
\draw [line width=2.pt,color=ffqqqq] (0.,2.5)-- (2.5,2.5);
\draw [line width=2.pt,color=ffqqqq] (2.5,2.5)-- (2.5,0.);
\draw [line width=2.pt,color=ffqqqq] (2.5,0.)-- (1.,0.);
\draw [line width=2.pt,color=ffqqqq] (1.,0.)-- (0.,1);
\node[label=right:{$ {\Delta } $}] (n1) at (0.75,1.) {};
\end{axis}
\end{tikzpicture}
\end{center}
\end{multicols}
\end{example}

\section{Altmann-Hausen presentation for normal affine $\R$-varieties} \label{SectionRAHP}

\subsection{Equivariant toric downgrading} \label{SubsectionETD}

Given an action of a complex torus $\T$ on a normal affine $\C$-variety $X$, Altmann and Hausen indicate in  \cite[\S 11]{Alt} a recipe  {on}   {how} to determine a semi-projective variety  $Y_X$ and a pp-divisor $\D_X$ mentioned in Theorem \ref{thmAlt}. The idea is to embed $\T$-equivariantly $X$ into a toric variety $\AC^n$    such that $X$   intersects  {the}  dense open  {orbit} of $\AC^n$ for the natural $\GC^n$-action. They construct a normal semi-projective variety $Y$ and a pp-divisor $\D$ describing the $\T$-action   on $\AC^n$. From these data, they obtain $Y_X$ and $\D_X$ describing the $\T$-action  on $X$. 

In this section, we describe a $\T$-equivariant  embedding $X \hookrightarrow \AC^n$ which is also $\Gamma$-equivariant (Proposition \ref{GaloisToricDowngrading}), and we use this embedding  to  {extend} the Altmann-Hausen presentation to the case of real torus actions on  affine $\R$-varieties (Theorems \ref{ResultT} and \ref{ResultT2}).

\begin{proposition}  \label{GaloisToricDowngrading}   
Let $X$ be an affine $\C$-variety    endowed with an  action of $\T$,  let $M:= \Hom_{gr}(\T, \GC)$, and let $d$ be the rank of $M$.  Let $\sigma$ be an $\R$-structure on $X$, and let $\tau$ be an $\R$-group structure on $\T$. If the real torus $(\T, \tau)$ acts on $(X, \sigma)$, then there exist  $n \in \N, n \geq d$ such that:
\begin{enumerate}[leftmargin=0.75cm, label=(\roman*)]
\item There is an  $\R$-group  structure $ \tau' $ on $\GC^n $   that extends to an  $\R$-structure $ \sigma' $ on $\AC^n $; 
\item $(\T, \tau)$ is a closed  {subgroup} of $(\GC^n,  {\tau}')$; and
\item $(X, \sigma)$ is a closed subvariety of   $(\AC^n,  {\sigma}')$   and   $(X, \sigma) \hookrightarrow (\AC^n,  {\sigma}')$ is  $(\T, \tau)$-equivariant. Moreover, $X$ intersects  the dense open orbit of   $\AC^n$  for the natural $\GC^n$-action, and the weight cone of $\AC^n$ is the weight cone of $X$.
\end{enumerate}
\end{proposition}

\begin{proof}
 \textit{(i)} The algebra $\C[X]$ is finitely generated, so we can write  $\C[X] = \C[\tilde{g}_1,  \dots , \tilde{g}_k]$ with $\tilde{g}_i \in \C[X] \backslash \{ 0 \}$. Since  $\C[X] = \bigoplus_{m \in M} \C[X]_m$, there exists homogeneous elements $\tilde{g}_{i,j}$ such that $\tilde{g}_i = \tilde{g}_{i, 1} + \dots + \tilde{g}_{i, k_i}$,  hence $\C[X] = \C[\tilde{g}_{i,j}]$. Note that $\C[X]= \C[\tilde{g}_{i,j}, {\sigma}^{\sharp}(\tilde{g}_{i,j}) ]$.  Moreover, by Lemma \ref{IsomGradedStructure}, an homogeneous element is send to an homogeneous element by ${\sigma}^{\sharp}  $. Hence we can assume that there exists $n \in \N$ such that $\C[X] = \C[ {g}_1,  \dots ,  {g}_n]$,  where the $g_i$ are homogeneous of degree $m_i \in M$ and such that the set $\{ g_i \ | \ 1 \leq i \leq n\}$ is stable under the involution   ${\sigma}^{\sharp}$. 
Let $\tau' $ and $ {\sigma'} $ be the maps induced by the  antilinear maps     
${\tau'}^{\sharp} (x_i) = x_j$, 
	  and $ {\sigma'}^{\sharp}  (x_i) = x_j$, where ${\sigma}^{\sharp}(g_i) = g_j$.  This induces  an $\R$-group structure on $\GC^n = \Spec( \C[x_1^{\pm1}, \dots, x_n^{\pm1}])$ and a   $\R$-structure  on $\AC^n = \Spec( \C[x_1, \dots, x_n])$.

\smallskip

 \textit{(ii)} The $\C$-algebra  morphism
$ \psi :  \C[x_1^{\pm 1},  \dots , x_n^{\pm 1}]    \to  \C[M],  \ 
 x_i   \mapsto \chi^{m_i}$ 
is surjective since the $\T$-action    on $X$ is effective.  Since $(\T, \tau)$ acts on $(X, \sigma)$, 
$\psi$ is $\Gamma$-equivariant. So,  the   $\R$-algebra morphism   $\psi^{\Gamma} : \C[x_1^{\pm 1},  \dots , x_n^{\pm 1}]^{\Gamma}    \to  \C[M]^{\Gamma}$ is well defined and surjective.  Hence, $(\T, \tau)$ is a closed subgroup of $(\GC^n,  {\tau'})$. 

\smallskip

 \textit{(iii)} The $\C$-algebra morphism   
$\varphi : \C[x_1,  \dots , x_n]    \to \C[X], \
 x_i   \mapsto g_i
$
is surjective and induces a $\C$-algebra isomorphism  $\C[g_1,  \dots , g_n] \cong  {\C[x_1,  \dots , x_n]}/{\mathfrak{a}}$, with $ \mathfrak{a} = \text{Ker}(\varphi)$.   {Moreover}, the morphism $\varphi$ is $\Gamma$-equivariant. So,  the   $\R$-algebra morphism  $\varphi^{\Gamma} : \C[x_1,  \dots , x_n]^{\Gamma}    \to \C[X]^{\Gamma}$ is well defined and surjective. Hence,   $(X, \sigma)$ is a closed subvariety of   $(\AC^n,  {\sigma'})$.

 Note that $\varphi$ is $\T$-equivariant, so the closed immersion    $X \hookrightarrow \AC^n$ is $\T$-equivariant. 
Moreover, the comorphism of   the  $\T$-action   on $\AC^n$ is given by:
\begin{equation*}
\tilde{\mu}^{\sharp}  :\C[x_1, \dots, x_n]      \to \C[M] \otimes  \C[x_1, \dots, x_n],   \ \ \ 
 x_i  \mapsto \chi^{m_1} \otimes x_i
\end{equation*}
Then,  the  following diagram  commutes:
\begin{center}
\begin{tikzpicture}
\matrix (m) [matrix of math nodes,row sep=1.25em,column sep=4em,minimum width=2em]
 {
              & \C[\AC^n] \vphantom{\otimes \C[\AC^n]} &                   &  \C[M] \otimes \C[\AC^n] \\
 \C[\AC^n] \vphantom{\otimes \C[\AC^n]}  &              & \C[M] \otimes \C[\AC^n] &      \\
             &              &                   &        \\
				     &              &                   &        \\
	           &       \C[X]  \vphantom{  \C[M] \otimes}  &                   & \C[M] \otimes \C[X] \\
      \C[X]  \vphantom{  \C[M] \otimes}     &              & \C[M] \otimes \C[X] &     \\};
 \path[-stealth]
    (m-1-2) edge node  [above] {   {$\tilde{\mu}^{\sharp}$} } (m-1-4)
    (m-2-1) edge node  [above] { \ \ \ \ \ \ \ \ \ \ \ \ \ \ \  {$\tilde{\mu}^{\sharp} $ } } (m-2-3)
	  (m-5-2) edge [dashed] node  [below] {  {$\mu^{\sharp} \ \ \ \ \ \ \ \ \ \ \ \ \ \ \ \  $} } (m-5-4)
    (m-6-1) edge node  [below] {  {$\mu^{\sharp} $} } (m-6-3)
		(m-1-4) edge node  [left] { {${\tau}^{\sharp} \times {\sigma'}^{\sharp} $}  \ \ \ \  } (m-2-3)
    (m-5-4) edge node  [right] { \ \ \ \  {${\tau}^{\sharp} \times  {\sigma}^{\sharp} $}} (m-6-3)
    (m-5-2) edge [dashed] node  [above] { {${\sigma}^{\sharp} \ \ \ \ \ $}}  (m-6-1)
		(m-1-2) edge node  [above] { {${\sigma'}^{\sharp} \ \ \ \ \ \ \ $}}  (m-2-1)
    (m-2-1) edge node  [left] { {$\varphi$}} (m-6-1)
    (m-1-2) edge [dashed] node  [left] {    {$\varphi$} } (m-5-2)
	  (m-1-4) edge   node  [right] {  $id \times \varphi$ }(m-5-4)
		(m-2-3) edge node  [right] {   ${id \times \varphi}_{\vphantom{N_N} }$ }   (m-6-3);
\end{tikzpicture}
\end{center}
Hence, the morphism $\varphi$ is $(\T, \tau)$-equivariant, so $(X, \sigma)$ is a closed subvariety of $(\AC^n,  {\sigma'})$, and $(X, \sigma) \hookrightarrow (\AC^n,  {\sigma'})$ is $(\T, \tau)$-equivariant.  

{Finally},  {note}  {that}  {for}  {all} $i \in \{1 , \dots, n \}$, $x_i \notin \mathfrak{a}$, hence $X$  {intersects}  {the}  {dense} open {orbit} {of} $\GC^n$. It follows that  the weight cone of $\AC^n$ is the weight cone of $X$. 
\end{proof}

\begin{example} \label{RWA3''}  We pursue Example  \ref{RWA3'}. The action of $(\GC^2, \tau_2)$ on $(\AC^3, \sigma')$ comes from the $\Gamma$-equivariant inclusion of $(\GC^2, \tau_2)$ in $(\GC^3, \tau')$ given by $(s,t) \mapsto (s,t,st)$, where $\tau'$ is the $\R$-group structure defined by $\tau' = \tau_2 \times \tau_0$ (see Example  \ref{RW3t}). We denote by $M$ and $M'$ the character lattices of $\GC^2$ and $\GC^3$ respectively. Then, we obtain the diagrams of Remark \ref{FracsectionR} with:
\begin{multicols}{5}
\footnotesize{\begin{center}
$F := \begin{bmatrix}
1 & 0 \\
0 & 1 \\
1 & 1 \\
\end{bmatrix}$
\end{center}

\begin{center}
{ \ \\ 
$P := \begin{bmatrix}
-1 & -1 & 1 \\
\end{bmatrix}$}
\end{center}

\begin{center}
{ \ \\
\vspace{-0.22cm}
$\hat{\tau}_2 := \begin{bmatrix}
0 & 1 \\
1 & 0 \\
\end{bmatrix}$}
\ \\
\end{center}

\begin{center}
$\hat{\tau}' := \begin{bmatrix}
0 & 1 & 0 \\
1 & 0 & 0 \\
0 & 0 & 1 \\
\end{bmatrix}$
\end{center}

\begin{center}
{ \ \\
$\hat{\tau}_Y := \begin{bmatrix}
1 \\
\end{bmatrix}$}
\end{center}}
\end{multicols}
\end{example}

\subsection{Real torus actions on  normal affine $\R$-varieties} \label{SubsectionRAH}

We present the main theorical results of this article concerning the  presentation of  affine $\R$-varieties endowed with  real torus actions:

\begin{theorem}  \label{ResultT}   
Let $(\T, \tau)$ be a real torus, let $M:= \Hom_{gr}(\T, \GC)$,   and 
  let $ (Y, \sigma_Y) $ be a   normal semi-projective $\R$-variety. 
 Let $\omega_{N}$ be a pointed {cone} in $N_{\Q}$, and let $\D$ be an $\omega_{N}$-pp-divisor on $Y$.   
Assume that there exists a  {monoid}   {morphism} $h :   \omega_{N}^{\vee} \cap M \to \C(Y)^*$ such that 
\begin{equation}  \let\veqno\eqno \label{eqdiv}
\forall m \in \omega_{N}^{\vee} \cap M, \  \sigma_Y^*(\D(m)) = \D(\tilde{\tau}(m)) + \Div_Y(h(\tilde{\tau}(m)))  \text{ \ \  and \ \ } h(m) \sigma_Y^{\sharp}(h(\tilde{\tau}(m)))=1,
\end{equation}
then there exists an $\R$-structure $\sigma_{X[Y, \D]}$ on the normal affine variety $X{[Y, \D]} $ such that  $(\T, \tau)$ acts   on $({X[Y, \D]}, \sigma_{X[Y, \D]})$.
\end{theorem}

\begin{remark}  The datum $(Y, \D)$ is used to construct the affine $\T$-variety $X[Y, \D]$. The monoid morphism $h$ is the additional datum that encodes the real structure $\sigma_{X[Y, \D]}$ on $X[Y, \D]$ such that $(X[Y, \D], \sigma_{X[Y, \D]})$ is a $(T, \tau)$-variety. 
\end{remark}

\begin{remark} \label{FunctionField}
Since $Y$ is an integral scheme, for any affine open subset $U \subset Y$, the ring  $\mathcal{O}_Y(U)$ is   integral   and $\C(Y)=\Frac(\mathcal{O}_Y(U))$. Hence, $\sigma_Y$ induces an $\R$-field automorphism denoted $\sigma_Y^{\sharp} : \C(Y) \to \C(Y)$.  The field of invariant rational functions is denoted by   $\C(Y)^{\Gamma}:= \{ f \in \C(Y) \ | \ \sigma_Y^{\sharp}(f)=f \}$.  A classical result, due to Artin, states  that if $G$ is a finite group of automorphisms  of a field $\Bbbk$, then $G=\Gal(\Bbbk/\Bbbk^G)$.   So,  the extension  $\C(Y) / \C(Y)^{\Gamma}$ is Galois, with Galois group $\Gamma$.  
\end{remark}

\begin{proof}
By   Theorem \ref{thmAlt} (1),  ${X[Y, \D]}:=\Spec(A[Y, \D])$ is a normal affine  $\C$-variety   endowed with a $\T$-action, of   weight cone   $\omega_{N}^{\vee}$. This action is obtained from the following comorphism:
\begin{equation*}
\mu^{\sharp} : A[Y, \D]  \to \C[M ] \otimes  A[Y, \D], \  
f \mathfrak{X}_m   \mapsto   \chi^{m} \otimes f \mathfrak{X}_m.
\end{equation*} 
We  now construct an $\R$-structure on ${X[Y, \D]}$ such that $(\T, \tau)$ acts on $({X[Y, \D]}, \sigma_{X[Y, \D]})$. Condition  (\ref{eqdiv}) implies that,  for all $m \in \omega_{N}^{\vee} \cap M$,  
\begin{equation*}
\alpha_m : \H^0( Y, \mathcal{O}_Y(\D({m}))) \mathfrak{X}_{ m}   \to  \H^0( Y, \mathcal{O}_Y(\D(\tilde{\tau}({m})))) \mathfrak{X}_{\tilde{\tau}({m})}, \ f \mathfrak{X}_m   \mapsto \sigma_Y^{\sharp}(f) h(\tilde{\tau}({m}))\mathfrak{X}_{\tilde{\tau}({m})}
\end{equation*}
are   isomorphisms of $A[Y, \D]_0$-modules and   
these isomorphisms collect into an involution $\oplus_{m \in \omega_{N}^{\vee} \cap M}  \alpha_m$ on the direct sum $A[Y, \D]$. The latter corresponds to a  $\R$-structure $\sigma_{X[Y, \D]}$ on ${X[Y, \D]}$. \\
Finally, $(\T, \tau)$ acts on $({X[Y, \D]}, \sigma_{X[Y, \D]})$ since the following diagram commutes:
\begin{center}
\begin{tikzpicture}
  \matrix (m) [matrix of math nodes,row sep=1.5em,column sep=7em,minimum width=2em]
  {
	  A[Y, \D]  &  \C[M] \otimes  A[Y, \D]        \\
	  A[Y, \D]  &  \C[M] \otimes  A[Y, \D]        \\};
 \path[-stealth]
   	(m-1-1) edge   node [above] { $\mu^{\sharp}$}  (m-1-2)
		(m-2-1) edge   node [below] { $\mu^{\sharp}$}  (m-2-2)
		(m-1-1) edge node [left] { $  {{\sigma^{\sharp}}_{X[Y, \D]} }$ }   (m-2-1)
		(m-1-2) edge node [right] { $  {\tau }^{\sharp} \otimes {{\sigma_{X[Y, \D]}} }^{\sharp} $ }   (m-2-2)
 ;
\end{tikzpicture}
\end{center}
\end{proof}

\begin{theorem} \label{ResultT2}   
Let $(\T, \tau)$ be a real torus and let $M:= \Hom_{gr}(\T, \GC)$. Let $(X, \sigma_X)$ be a normal   affine $\R$-variety endowed with a   $(\T  , \tau)$-action.     Let $\omega_{N}$  be the cone in $N_{\Q}$ dual to the weight cone $ \omega_{M}$.    There exists a normal semi-projective $\R$-variety   $(Y, \sigma_Y)$,    an $\omega_{N}$-pp-divisor  $\D$   on $Y$, and  a  {monoid}  {morphism} $h :   \omega_{M}  \cap M \to \C(Y)^*$ such that  
$$\forall  m \in \omega_{M} \cap M, \ \sigma_Y^*(\D(m)) = \D(\tilde{\tau}(m)) + \Div_Y(h(\tilde{\tau}(m))) \text{ \ \  and \ \ } h(m) \sigma_Y^{\sharp}(h(\tilde{\tau}(m)))=1,$$
 and such that the affine varieties $(X, \sigma_X)$ and $({X[Y, \D]}, \sigma_{X[Y, \D]})$ are $(\T, \tau)$-equivariantly isomorphic.  
\end{theorem}

\begin{proof} $\bullet$ \textit{\textbf{Step 0}:   Preliminaries.}

Using Proposition  \ref{GaloisToricDowngrading}, there exists $n \in \N$ such that $(\T, \tau)$ is a closed subgroup of $(\GC^n, \tau')$ and $(X, \sigma_X)$ is a closed $(\T, \tau)$-equivariant subvariety of $(\AC^n, \sigma)$. So, let $\mathfrak{a}$ be the ideal of $\C[\AC^n] = \C[x_1, \dots, x_n]$ such that $\C[X]$ is $(\Gamma \times \T)$-equivariantly isomorphic to $\C[\AC^n]/\mathfrak{a}$. We write $\C[X] =  \C[\AC^n]/\mathfrak{a}$. Let $M':=\Hom_{gr}(\GC^n,\GC)$ and let $M_Y$ be the sublattice of $M'$ constructed in Remark \ref{FracsectionR}.  We have the   commutative diagrams of Remark \ref{FracsectionR}. 
Recall that there always exists a section $s^* : M \to M'$, but not necessarily $\Gamma$-equivariant, and a cosection $t^* : M' \to M_Y$. These homomorphisms satisfy $F^* \circ s^* = Id_{M}$, $t^* \circ P^* = Id_{M_Y}$ and $P^* \circ t^* = Id_{M'} - s^* \circ F^*$.

Since $\Frac(\C[M']) = \C(\AC^n)$, the section $s^*$ induces a  morphism
$$u : \omega_{M} \cap M \to   \C(\AC^n)^*, m \mapsto \chi^{s^*(m)}$$
such that for all ${m} \in \omega_{M} \cap M$, $u({m}) \in \C(\AC^n)_{m}$.   

 Note that, for all $  m \in \omega_M  \cap M $, $\C(\AC^n)_m =  \C(\AC^n)_0 u(m)$.
Indeed, since $\C[\AC^n]_m \subset \C(\AC^n)_m$, we can write  $\C[\AC^n]_m = \widetilde{\C[\AC^n]}_m u(m)$, 
with $\widetilde{\C[\AC^n]}_m \subset \C(\AC^n)_0$.
Then , we can write:
$$\C[\AC^n] = \bigoplus_{{m} \in \omega_{M} \cap M} \widetilde{\C[\AC^n]}_{m} u({m}) \subset \C(\AC^n)_0[M].$$ 
Since {for} {all} $i$, $x_i \notin \mathfrak{a}$, for all ${m} \in \omega_{M} \cap M$ {we} {have}   $u_X({m}) \in \C(X)_{m}$, where $u_X(m)$ is obtained from the surjective morphism $\C[\AC^n] \twoheadrightarrow \C[X]$. It follows a morphism $u_X : \omega_{M} \cap M \to \C(X)^*$. Hence we can write
$$\C[X] = \bigoplus_{{m} \in \omega_{M} \cap M } \C[X]_{{m}} = \bigoplus_{{m} \in  \omega_{M} \cap M } \widetilde{\C[X]}_{m}u_X({m}) \subset \C(X)_0[M],$$
where, for all $m \in \omega_M \cap M$, $\C[X]_m = \widetilde{\C[X]}_m u_X(m)$ and $\widetilde{\C[X]}_m \subset \C(X)_0$.

\smallskip

$\bullet$ \textit{\textbf{Step 1}: Altmann-Hausen quotient and divisors. } 
 
Let $\{e_1, \dots, e_n \}$ be the standard basis of $N'$. The cone   in $  N'_{\Q}$ of the toric variety $\AC^n$ is $\Q_{\geq 0}^n$, and    $F^*(\Q^n_{\geq 0}) = \omega_{M}$.  
 Since the fan in $N'_{\Q}$ generated by $\{e_1, \dots, e_n \}$  is $\Gamma$-stable (for $\hat{\tau}'$) and $P$  is $\Gamma$-equivariant,  the fan  $\Lambda_Y$ in $(N_Y)_{\Q}$ generated by $\{P(e_1), \dots, P(e_n) \}$  is $\Gamma$-stable  (for $\hat{\tau}_Y$).

Let $Y$ be the toric variety obtained from the fan $\Lambda_Y$; it is a  semi-projective variety  (see \cite[Proposition 7.2.9]{Cox}). 
Since  $\Lambda_Y$  is $\Gamma$-stable, the $\R$-group structure $\tau_Y$ on $\T_Y := \Spec(\C[M_Y])$  extends to an $\R$-structure $\sigma_Y$ on $Y$ by   \cite[Proposition 1.19]{Huruguen}.

Let $Y_X$ be the closure of the image of $X \cap \GC^n$ in $Y$ by the  surjective group homomorphism $\pi : \GC^n \twoheadrightarrow \T_Y$ composed with the inclusion $\T_Y \hookrightarrow Y$. Since these morphisms are $\Gamma$-equivariant, the $\R$-structure on $Y$  restricts to an $\R$-structure $\sigma_{Y_X}$ on  $Y_X$.    The normalization $\tilde{Y}_X$ of $Y_X$, with morphism $\nu : \tilde{Y}_X \to Y_X$,   is a  semi-projective variety.  Using universal property of normalization and the fact that $\nu$ is an isomorphism on a dense  open subset of $Y_X$, there exists an $\R$-structure $\sigma_{  \tilde{Y}_X }$ on $\tilde{Y}_X$ which   makes the following diagram commute: 
\vspace{-0.25cm}
	\begin{center}
	\begin{tikzpicture}
		\matrix (m) [matrix of math nodes,row sep=1.25em,column sep=5em,minimum width=2em]
  {
	   \tilde{Y}_X &   \tilde{Y}_X        \\
	     {Y}_X &    {Y}_X     \\};
 \path[-stealth]
   	(m-1-1) edge   node [above] {$\sigma_{  \tilde{Y}_X } $ }  (m-1-2)
		(m-2-1) edge   node [above] { $\sigma_{   {Y}_X } $}  (m-2-2)
		(m-1-1) edge   node [left] {$\nu$}   (m-2-1)
		(m-1-2) edge   node [right] {$\nu$}  (m-2-2);
	\end{tikzpicture}
\end{center}
\vspace{-0.25cm}
For each ray   of the fan $\Lambda_Y$,  we denote by $v_i$ its first lattice vector.  
 To a ray spanned by $v_i $   corresponds  a toric divisor $D_{v_i}$ on $Y$. 

 The divisor $\D = \sum_{v_i} \Delta_{v_i} \otimes D_{v_i}$, where $\Delta_{v_i} := s(\Q^n_{\geq 0} \cap P^{-1}(v_i))$, is an $\omega_{N}$-pp-divisor on $Y$. Let ${\D}_X$ be the divisor obtained by pulling back   $\D$ to $\tilde{Y}_X$. It  is an $\omega_{N}$-pp-divisor on $\tilde{Y}_X$ (see \cite[Proposition 8.1]{Alt03}  which is not in the published version).

\smallskip

$\bullet$  \textit{\textbf{Step 2}: Isomorphisms  $\C(Y)     \cong \C(\AC^n)_0   \text{ \ and \ } \C(Y_X)  \cong  \C(X)_0$.} 
 
Observe that $\pi^{\sharp} : \C[M_Y] \to \C[M']_0$ is an isomorphism.   Hence,   $\pi^{\sharp}$ induces an isomorphism:  
$$    \Frac(\C[M_Y])   \to \Frac(\C[M']_0 ),   \  
\dfrac{f}{g} = \dfrac{\sum a_i \chi^{m_i} }{ \sum b_j \chi^{m_j} }   \mapsto \dfrac{\pi^{\sharp}(f)}{\pi^{\sharp}(g)} = \dfrac{\sum a_i \chi^{P^*(m_i)} }{ \sum b_j \chi^{P^*(m_j)} }.$$
For all $m \in M$, note that  $(F^*)^{-1}(m) = s^*(m) + \text{Ker}(F^*)$.  Hence,  $\Frac(\C[M'])_0 = \Frac(\C[M']_0)$. Since $\T_Y$ is a dense   open  subset of $Y$, we have  $\C(Y)=\C(\T_Y)$. Since $\GC^n$ is  a   dense  open subset of $\AC^n$, we have $\Frac(\C[M'])= \C(\AC^n)$.  Therefore, $\Frac(\C[M'])_0 = \C(\AC^n)_0$. Finally, we obtain  a $\Gamma$-equivariant isomorphism: 
\begin{equation*}
\varphi :  \C(Y)    \to \C(\AC^n)_0. 
\end{equation*} 
Moreover, the inclusion $X \hookrightarrow \AC^n$ is $\T$-equivariant, the variety  $\T_Y \cap Y_X$ is affine and the following diagram   commutes:
\vspace{-0.25cm}
\begin{center}
\begin{tikzpicture}
  \matrix (m) [matrix of math nodes,row sep=1.25em,column sep=5em,minimum width=2em]
  {
	   \GC^n  &   \T_Y^{\vphantom{n}}          \\
	     \GC^n \cap X &    \T_Y \cap Y_X = \pi( \GC^n \cap X)  \\};
 \path[-stealth]
   	(m-1-1) edge  [->>] node [above] { $\pi $}   (m-1-2)
		(m-2-1) edge  [->>]  (m-2-2)
		(m-2-1) edge  [right hook-latex]    (m-1-1)
		(m-2-2) edge  [right hook-latex]   (m-1-2);
	\end{tikzpicture}
\end{center}
\vspace{-0.25cm}
Therefore the following diagram commutes:
\vspace{-0.25cm}
\begin{center}
\begin{tikzpicture}
\matrix (m) [matrix of math nodes,row sep=1.25em,column sep=6em,minimum width=2em]
 {
    \C[{\T_Y\cap Y_X}]^{\vphantom{n}}  &      \C[{\GC^n\cap X} ]_0     \\
			\C[{\T_Y   }]     &   \C[\GC^n ]_0    \\};
\path[-stealth]
  (m-1-1) edge [right hook->>]      (m-1-2)
	(m-2-1) edge [right hook->>] node [above] {$\pi^{\sharp}$}  (m-2-2)
	(m-2-1) edge  [->>]   (m-1-1)
	(m-2-2) edge  [->>]   (m-1-2);
\end{tikzpicture}
\end{center}
\vspace{-0.25cm}
It follows an isomorphism $\C(\T_Y \cap Y_X) \to   \Frac(\C[\GC^n \cap X]_0)$. Since,  
$$\Frac(\C[\GC^n \cap X]_0) = \C(\GC^n \cap X)_0 , \   \C(Y_X) = \C(\T_Y \cap Y_X), \text{  and } \C(X) = \C(\GC^n \cap X),$$  
we obtain a $\Gamma$-equivariant isomorphism $   \C(Y_X)    \cong \C(X)_0 $. 
Note that we have a $\Gamma$-equivariant isomorphism $\C(Y_X) \cong \C(\tilde{Y}_X)$, therefore, the isomorphism:  
 $$ {\varphi}_X :  \C(\tilde{Y}_X)    \to \C(X)_0$$
is $\Gamma$-equivariant.

\smallskip

$\bullet$  \textit{\textbf{Step 3}: Isomorphisms  $A[ {Y} ,  {\D} ] \cong \C[\AC^n]$ and $A[\tilde{Y}_X, {\D}_X] \cong \C[X]$.}

  Let $m \in  \omega_{M} \cap M$. Consider  the polyhedron $\Delta(m)     := (F^*)^{-1}(m) \cap \Q^n_{\geq 0} \subset M'_{\Q}$  and the polyhedron  $\Delta_Y(m)  := t^*(\Delta(m) ) \subset (M_Y)_{\Q}$.  
Note that  
$$\widetilde{\C[\AC^n]}_{m} := \bigoplus_{m'  \in  \Q^n_{\geq 0} \cap(F^*)^{-1}(m)   \cap M'} \C \chi^{m'} = \bigoplus_{m_Y \in  \Delta_Y(m) \cap M_Y} \C \chi^{P^*(m_Y) } = \varphi \left(  \bigoplus_{m_Y \in  \Delta_Y(m) \cap M_Y} \C \chi^{ m_Y } \right).$$
It follows from  Lemma   \cite[\S 3.4]{Fult} and   the proof of \cite[Proposition 8.5]{Alt03}  that:
\begin{equation*}
\H^0(Y, \mathcal{O}_Y(\D(m))) = \bigoplus_{m_Y \in  \Delta_Y( m  ) \cap M_Y} \C \chi^{ m_Y }.
\end{equation*}
Therefore, we obtain   a {graded} {algebra} {isomorphism}: 
\begin{align*}
{\Phi} : A[ {Y} ,  {\D} ]   \to \C[\AC^n],  \ \ 
f  \mathfrak{X}_m   \mapsto{\varphi}(f ) u({m}).
\end{align*}
Moreover, there is a natural surjective graded algebra morphism:
$$\Psi : A[ {Y} ,  {\D} ] \twoheadrightarrow A[\tilde{Y}_X, {\D}_X].$$
Let ${\Phi}_X : A[\tilde{Y}_X, {\D}_X]   \to \C(X)_0[M] $ be  the morphism defined by $
f  \mathfrak{X}_{m}   \mapsto  {\varphi}_X(f ) u_X({m})$. 
Since the following diagram commutes:
\vspace{-0.25cm}
\begin{center}
\begin{tikzpicture}
  \matrix (m) [matrix of math nodes,row sep=1.5em,column sep=6em,minimum width=2em]
  {
        A[Y,  \D]       & \C[\AC^n]    \\
	A[\tilde{Y}_X, \D_X]  &   \C[X]      \\};
 \path[-stealth]
   	(m-1-1) edge [right hook->>] node [above] {$\Phi$}  (m-1-2)
		(m-2-1) edge [->>] node [above] {$\Phi_X$}          (m-2-2)
		(m-1-1) edge [->>] node [left]  {$\Psi$}            (m-2-1)
		(m-1-2) edge [->>]        (m-2-2);
	\end{tikzpicture}
\end{center}
\vspace{-0.25cm}
 we obtain  a {graded} {algebra} isomorphism: 
\begin{align*}
{\Phi}_X : A[\tilde{Y}_X, {\D}_X]    \to \C[X],  \ \ 
f  \mathfrak{X}_{m}    \mapsto  {\varphi}_X(f ) u_X({m}).
\end{align*}

\smallskip

$\bullet$  \textit{\textbf{Step 4}: Equality   $\sigma_{\tilde{Y}_X}^* \left({\D}_X({m}) \right) = {\D}_X(\tilde{\tau}({m})) + \Div(h(\tilde{\tau}({m})))$,       for all ${m} \in \omega_{M}  \cap M $. } 
 
Let ${m} \in \omega_{M}  \cap M$. Let $h': \omega_M \cap M \to  \C(\AC^n)_0$ be the monoid morphism defined by $h'(m) := \dfrac{\sigma^{\sharp}(u(\tilde{\tau}(m))}{u(m)}$.  By Lemma  \ref{IsomGradedStructure}, $\sigma^{\sharp}(\C[\AC^n]_{\tilde{\tau}({m}) }) = \C[\AC^n]_{m}$. It follows $h'(m) \in   \widetilde{\C[\AC^n]}_{m}$. Moreover, note that $h'(m) \sigma^{\sharp} (h'(\tilde{\tau}({m})))=1$. Consider the monoid morphism $h := \tilde{\varphi}^{-1} \circ h' : \omega_{M}  \cap M \to \C(Y)^*$.   We construct  a  $\R$-structure  on $X[Y, \D]$ using the following commutative diagram: 
\vspace{-0.5cm}
\begin{multicols}{2}
\small{ \begin{flushleft}
\ \\
\begin{tikzpicture}
  \matrix (m) [matrix of math nodes,row sep=1.75em,column sep=2em,minimum width=1.25em]
  {
	   \C[\AC^n] &   \C[\AC^n]          \\
	     A[Y, \D] &   A[Y, \D]    \\};
 \path[-stealth]
   	(m-1-1) edge   node [above] { $\sigma^{\sharp}$}   (m-1-2)
		(m-2-1) edge     (m-2-2)
		(m-2-1) edge    node [ left] { $ {\Phi} $ }   (m-1-1)
		(m-1-2) edge    node [right] { $ {\Phi}^{-1} $ }  (m-2-2);
	\end{tikzpicture}
\end{flushleft}

\begin{flushleft}
\begin{tikzpicture}
 \matrix (m) [matrix of math nodes,row sep=1.25em,column sep=1em,minimum width=1.25em]
 {
\varphi(f)u(m) & {\sigma}^{\sharp} \Big(\varphi(f)\Big)h'(\tilde{\tau}(m)) u(\tilde{\tau}(m))   \\
 f \mathfrak{X}_{m} & \varphi^{-1} \Big(\sigma^{\sharp} \Big(\varphi(f)\Big) h'(\tilde{\tau}(m)) \Big) \mathfrak{X}_{\tilde{\tau}(m)}  \\ };
\path[-stealth]
  (m-1-1) edge [|->] node [above] { \vphantom{$\sigma^{\sharp}$} } (m-1-2)
	(m-2-1) edge [|->] (m-2-2)
	(m-2-1) edge [|->] (m-1-1)
	(m-1-2) edge [|->] (m-2-2);
\end{tikzpicture}
\end{flushleft}}
\end{multicols}
\vspace{-0.25cm}
Since $ {\varphi}$ is $\Gamma$-equivariant, we have $ \varphi^{-1} \Big(\sigma^{\sharp} \Big(\varphi(f)\Big) h'(\tilde{\tau}(m)) \Big)    =       \sigma_{Y}^{\sharp}(f)   h(\tilde{\tau}(m))    $.  Hence,  the morphism 
$$A[Y, \D]  \to   A[Y, \D],  \ 
f \mathfrak{X}_{m}  \mapsto {\sigma_{Y}}^{\sharp}(f) h(\tilde{\tau}({m}))\mathfrak{X}_{\tilde{\tau}({m})}$$
induces an   $\R$-structure $\sigma_{X[Y,\D]}$ on  $X[Y,\D]$. 
From this we deduce  that: 
\begin{equation} \label{eq1} \let\veqno\eqno
\sigma_{Y}^*(\D({m})) = \D(\tilde{\tau}({m})) + \Div_Y(h(\tilde{\tau}({m}))). 
\end{equation}
 Moreover, $ h(m) \sigma_{Y}^{\sharp} (h(\tilde{\tau}({m})))=1$.
By the same reasoning, we construct an   $\R$-structure $\sigma_{X[\tilde{Y}_X,\D_X]}$ on  $X[\tilde{Y}_X,\D_X]$ from the following morphism:
$$A[\tilde{Y}_X,\D_X]  \to   A[\tilde{Y}_X,\D_X],  \ 
f \mathfrak{X}_{m}  \mapsto {\sigma_{\tilde{Y}_X}}^{\sharp}(f) h_X(\tilde{\tau}({m}))\mathfrak{X}_{\tilde{\tau}({m})},$$
where $h_X$ is obtained from $h$ by the projection $\C[\AC^n] \twoheadrightarrow \C[X]$.
Thus we obtain:
$$\sigma_{\tilde{Y}_X}^* \left(\D_X({m}) \right) = \D_X(\tilde{\tau}({m})) + \Div_{\tilde{Y}_X}(h_X(\tilde{\tau}({m}))) $$ and $h_X(m) \sigma_{Y}^{\sharp} (h_X(\tilde{\tau}({m})))=1$. 
\end{proof}

\begin{remark} \label{equivcocyvle} The construction of the real   variety $(X[Y, \D], \sigma_{X[Y, \D]})$ does not depend on the choice of the cosection. Indeed, for $j \in \{ 1, 2 \}$, let $s_j : N' \to N$ be two cosections, let $\D_j := \sum_i     \Delta_{v_i}^j \otimes D_{v_i} $ be the two associated pp-divisors, and let $h_j : \omega_M \cap M \to \C(Y)^*$ be the monoid morphisms constructed in Step 4. Note that for all $m \in M$, $(s_1^*-s_2^*)(m) \in \text{Ker}(F^*)=\text{Im}(P^*)$, thus there exists a lattice homomorphism $s_0 : N_Y \to N$ such that $s_1-s_2= s_0 \circ P$. Let $g : M \to \C(Y)^*$ be the morphism defined by $g(m):= \chi^{s_0^*(m)}$. Let $m \in \omega_M \cap M$. Since $\Delta_{v_i}^1 =  s_0   ( v_i ) + \Delta_{v_i}^2$, we have :
\begin{align*}
\D_1(m) =    \sum_i \langle s_0^*(m) |  v_i   \rangle \otimes  D_{v_i}    +   \D_2(m)  = \text{div}_Y(g(m)) + \D_2(m).
\end{align*}
Therefore the $M$-graded algebras $A[Y, \D_1]$ and $A[Y, \D_2]$ are isomorphic via:
$$A[Y, \D_1] \to A[Y, \D_2], f \mathfrak{X}_m \mapsto  f g(m) \mathfrak{X}_m.$$
Moreover, since
$$\dfrac{\sigma_Y^{\sharp}(g(\tilde{\tau}(m)))}{g(m)} =  \dfrac{\sigma_Y^{\sharp} \left(\chi^{s_0^*(\tilde{\tau}(m))} \right)}{\chi^{s_0^*( m )}} = \tilde{\varphi}_X^{-1} \left( \dfrac{\sigma_Y^{\sharp} \left(\chi^{P^* \circ s_0^*(\tilde{\tau}(m))} \right)}{\chi^{P^* \circ  s_0^*( m )}} \right) = \tilde{\varphi}_X^{-1} \left( \dfrac{\sigma_Y^{\sharp} \left( u_1 (\tilde{\tau}(m)) \right) u_2(m) }{ u_1(m) \sigma_Y^{\sharp} \left( u_2 (\tilde{\tau}(m)) \right)  } \right)  = \dfrac{h_1(m)}{h_2(m)},$$ 
the following diagram commutes:
\vspace{-0.25cm}
\begin{multicols}{2}
\small{ \begin{flushleft}
\begin{tikzpicture}
  \matrix (m) [matrix of math nodes,row sep=1.75em,column sep=4em,minimum width=1.25em]
  {
	   A[Y, \D_1] &   A[Y, \D_1]    \\
	   A[Y, \D_2] &   A[Y, \D_2]    \\};
 \path[-stealth]
   	(m-1-1) edge   node [above] { $\sigma_{X[Y,\D_1]}^{\sharp}$}   (m-1-2)
		(m-2-1) edge   node [above] { $\sigma_{X[Y,\D_2]}^{\sharp}$}  (m-2-2)
		(m-1-1) edge   node [ left] { $ \cong $ }   (m-2-1)
		(m-1-2) edge   node [right] { $ \cong $ }  (m-2-2);
	\end{tikzpicture}
\end{flushleft}

\begin{flushleft}
\begin{tikzpicture}
 \matrix (m) [matrix of math nodes,row sep=1.25em,column sep=1em,minimum width=1.25em]
 {
 f \mathfrak{X}_m & {\sigma}_Y^{\sharp}(f)  h_1(\tilde{\tau}(m)) \mathfrak{X}_{\tilde{\tau}(m)}   \\
 f g(m) \mathfrak{X}_{m} & {\sigma}_Y^{\sharp}(f g(m))  h_2(\tilde{\tau}(m)) \mathfrak{X}_{\tilde{\tau}(m)}   \\ };
\path[-stealth]
  (m-1-1) edge [|->] node [above] { \vphantom{$\sigma^{\sharp}$} } (m-1-2)
	(m-2-1) edge [|->] (m-2-2)
	(m-1-1) edge [|->] (m-2-1)
	(m-1-2) edge [|->] (m-2-2);
\end{tikzpicture}
\end{flushleft}}
\end{multicols}
\vspace{-0.25cm}
Hence, the varieties $({X[Y, \D_1]}, \sigma_{X[Y, \D_1]})$ and $({X[Y, \D_2]}, \sigma_{X[Y, \D_2]})$ are $(\T, \tau)$-equivariantly isomorphic.
\end{remark}

\subsection{Galois cohomology and real torus actions} \label{SubsectionH}

We recall some cohomological results    in view of  simplifying  the Altmann-Hausen presentation in the case where the real  acting torus is quasi-split.

Let $G$ be an abstract group equipped with a $\Gamma$-action denoted by $\star$.  A  \textit{cocycle} $a : \Gamma \to  {G}$     is  a map such that  $a_{id}=1$   and $a_{\gamma}   (\gamma \star a_{\gamma}) = 1$.    Two cocycles $a$ and $b$ are \textit{equivalent} if there exists $g \in G$ such that $b_{\gamma} = g^{-1}  a_{\gamma} ( \gamma \star g)$. The set of cocycles modulo this equivalence relation is the \textit{first pointed} set of Galois cohomology   $\H^1(\Gamma,  G)$.  If $G$ is an abelian $\Gamma$-group, then $\H^1(\Gamma,  G)$ is a group (see \cite{Serre}).

In the following result (Corollary \ref{ResultH} \textit{(ii)}),  we see that the Altmann-Hausen presentation simplifies if a certain cohomology set is trivial.  This simplification consists of choosing a pp-divisor on $Y$ such that $h=1$.

\begin{corollary}  \label{ResultH} Fix a real torus $(\T, \tau)$. Let $M:=\Hom_{gr}(\T, \GC)$. 
\begin{enumerate}[leftmargin=0.55cm, label=(\roman*)]
\item    Let $ (Y, \sigma_Y) $ be a normal semi-projective variety. 
 Let $\omega_{N}$ be a pointed {cone} in $N_{\Q}$ and $\D$ be an $\omega_{N}$-pp divisor on $Y$.    
If 
$$\forall m \in \omega_{N}^{\vee} \cap M, \  \sigma_Y^*(\D(m)) = \D(\tilde{\tau}(m)),$$
then there exists an $\R$-structure $\sigma_{X[Y, \D]}$ on the affine variety $X[Y, \D]$ such that  $(\T, \tau)$ acts  on $(X[Y, \D], \sigma_{X[Y, \D]})$.  
\item  Let $(X, \sigma_X)$ be an affine variety endowed with an action of $(\T  , \tau)$ of weight cone $\omega_{M} \subset M_{\Q}$, and let $\omega_{N}$ be the cone in ${N}$ dual to $ \omega_{M}$.    Let $(Y, \sigma_Y)$ be the $\R$-variety of Theorem \ref{ResultT2} and let $G:= \Hom_{gr}(M, \C(Y)^*)$ endowed with the $\Gamma$-action  $\gamma \star f := \sigma_Y^{\sharp} \circ f \circ \tilde{\tau}$.
If $\H^1( \Gamma, G) = \{ 1 \}$, then there exists  an $\omega_{N}$-pp divisor  $\D$  on $Y$ 
 such that  
$$\forall  m \in \omega_{M} \cap M, \ \sigma_Y^*(\D(m)) = \D(\tilde{\tau}(m)),$$
 and such that the varieties $(X, \sigma_X)$ and $({X[Y, \D]}, \sigma_{X[Y, \D]})$ are $(\T, \tau)$-equivariantly isomorphic.
\end{enumerate}
\end{corollary}

\begin{proof}
\textit{(i)}  We specify the proof of Theorem \ref{ResultT} with $h=1 \in \C(Y)^*$. \\ 
\textit{(ii)} 
By Theorem \ref{ResultT2}  there exists  a normal semi-projective variety   $Y$,    an $\omega_{N}$-pp-divisor  $\D$   on $Y$,  an $\R$-structure $\sigma_Y$ on $Y$ and a  {monoid}  {morphism} $h :   \omega_{M}  \cap M \to \C(Y)^*$ such that  $\sigma_Y^*(\D(m)) = \D(\tilde{\tau}(m) ) + \Div_Y(h(\tilde{\tau}(m) )) $    and    $h(m) \sigma_Y^{\sharp}(h(\tilde{\tau}(m))=1$ for all $m \in \omega_{M} \cap M$, and such that  the varieties $(X, \sigma_X)$ and $({X[Y, \D']}, \sigma_{X[Y, \D']})$ are $(\T, \tau)$-equivariantly isomorphic.
 Consider   $a : \Gamma \to    G$ defined by   $a_{id} : m \mapsto 1 \in \C(Y)^*$ and $a_{\gamma} : m \mapsto  h(m)  $. By construction, $a$   is a  cocycle.  Since $\H^1(\Gamma, G ) = \{ 1 \}$,  the  cocycle $b : \Gamma \to     G$ defined by $b_{id} = b_{\gamma} : m \mapsto 1 \in \C(Y)^*$  is equivalent to $a$, so there exists $g \in G$ such that $h(m) = g^{-1}(m) \sigma_Y^{\sharp}(g(\tilde{\tau}(m)))$. Let   $m \in \omega_{M} \cap M$, then:
$$\sigma_Y^*(\D(m)) = \D( \tilde{\tau}(m) ) + \Div_Y(h(\tilde{\tau}(m)) ) = \D( \tilde{\tau}(m) ) + \sigma_Y^*\Div_Y(g(m))  -  \Div_Y(g(\tilde{\tau}(m))) $$
$$\iff \sigma_Y^*(\D(m) -  \Div_Y(g(m)))= \D( \tilde{\tau}(m) )   -  \Div_Y(g(\tilde{\tau}(m))).$$
So, if $\D'$ is the  pp-divisor defined by $ \D'(m):= \D(m)-\text{div}_{Y}(g(m))$, then $\sigma_Y^*(\D'(m))  = \D'( \tilde{\tau}(m) )$ and the $M$-graded  algebras $A[Y, \D]$ and $A[Y, \D']$ are  isomorphic with respect to $\sigma_{X[Y,\D]}^{\sharp}$ and $\sigma_{X[Y,\D']}^{\sharp}$ (see the diagram of Remark \ref{equivcocyvle} with $\D_1 = \D$, $h_1=h$, $\D_2=\D'$ and $h_2=1$). Hence the varieties $(X, \sigma_X)$ and $({X[Y, \D']}, \sigma_{X[Y, \D']})$ are $(\T, \tau)$-equivariantly isomorphic.
\end{proof}

\begin{remark} 
The converse of Corollary \ref{ResultH} \textit{(ii)} is false; see Example    \ref{exc}.
\end{remark}

We    provide some situations with trivial Galois cohomology set making it possible to apply Corollary \ref{ResultH}  to simplify the Altmann-Hausen presentation. Let   $ \L  /  \K $ be a quadratic extension and denote by $\{ \varphi_{id}, \varphi_{\gamma} \}$ the Galois group $\Gal(\L/\K)$.

\begin{lemma} \label{H1}
Let $M  $ be a rank $n$ lattice     and  let $G_n:= \Hom_{gr}(M, \L^*)$. If $\Gal(\L/\K)$ acts on $G_n$ by  $\gamma \star f := \varphi_{\gamma} \circ f$,  
then $\H^1( \Gal(\L/\K), G_n) = \{ 1 \}$.
\end{lemma}

\begin{proof} 
We prove this result by induction on $n$. Let $n=1$, then $G_1 \cong \L^*$. The Galois group $\Gal(\L/\K)$ acts on $G_1 \cong \L^*$ and this action comes from the $\Gamma$-action on $\L$ defined by $\gamma \cdot z = \varphi_{\gamma}(z)$. By Hilbert's theorem 90, $\H^1(\Gal(\L/\K), G_1) = \{ 1 \}$. Let $n \geq 1$ and assume that the Lemma \ref{H1} is true for this fixed $n$.       We have a $\Gal(\L/\K)$-equivariant short exact sequence of $\Gal(\L/\K)$-groups: 
\begin{center}
\begin{tikzpicture}
\matrix (m) [matrix of math nodes,row sep=0.7em,column sep=5em,minimum width=2em]  {
	1_{\vphantom{n}}   & G_n   & G_{n+1} & G_1 & 1_{\vphantom{n}}    \\};
 \path[-stealth]
  	(m-1-1) edge    (m-1-2)
		(m-1-2) edge  node [above] { $\Psi$ }  (m-1-3)
		(m-1-3) edge  node [above] { $\Psi'$ }   (m-1-4)
		(m-1-4) edge    (m-1-5);
	\end{tikzpicture}
\end{center}
with 
$\Psi(f) :  \Z^n \oplus \Z \to \L^*, (k_1, \dots, k_n, k) \mapsto f(k_1, \dots, k_n)$ and $\Psi'(g) :    \Z \to \L^*, k \mapsto g(0, \dots, 0, k)$, where $f \in G_n$ and $g \in G_{n+1}$. This induces an exact sequence in Galois cohomology: 
 
\hspace{-1.2cm} 
\small 
\begin{tikzpicture}
\matrix (m) [matrix of math nodes,row sep=0.7em,column sep=0.75em,minimum width=2em]  {
1 & G_n^{\Gal(\L/\K)}   & G_{n+1}^{\Gal(\L/\K)} & G_1^{\Gal(\L/\K)} &  \H^1( \Gal(\L/\K), G_n) & \H^1( \Gal(\L/\K), G_{n+1}) & \H^1( \Gal(\L/\K), G_1)   \\};
 \path[-stealth]
  	(m-1-1) edge    (m-1-2)
		(m-1-2) edge    (m-1-3)
		(m-1-3) edge    (m-1-4)
		(m-1-4) edge    (m-1-5)
		(m-1-5) edge    (m-1-6)
		(m-1-6) edge    (m-1-7);
	\end{tikzpicture} 
\normalsize
By induction, $\H^1( \Gal(\L/\K), G_{n})  = \{ 1 \} $ and $\H^1( \Gal(\L/\K), G_{1})  = \{ 1 \} $. Therefore, 
$$\H^1( \Gal(\L/\K), G_{n+1})  = \{ 1 \}.$$
\end{proof}

\begin{lemma} \label{H1WR}
Let $M:=M_0 \oplus M_0$, where $M_0 \cong \Z^n$,  and let $\rho := \Gamma \to \GL(M)$ be the representation which exchanges the two factors.  Let $G= \Hom_{gr}(M, \L^* )$ be endowed with the $\Gamma$-action defined by $\gamma \star f := \varphi_{\gamma} \circ f \circ \rho(\gamma^{-1})$,
Then $\H^1( \Gamma, {G}) = \{ 1 \}$.	
\end{lemma}

\begin{proof} 
A cocycle is  uniquely determined by a homomorphism $h=a_{\gamma} : M \to \L^*$ which satisfies $h(m, m') \varphi_{\gamma}(h(m', m)) =1$ (the constant homomorphism) for every $(m, m') \in M=M_0 \oplus M_0$. In particular, we have $\varphi_{\gamma}(h(-m',0))h(0,-m')=1$, hence $\varphi_{\gamma}(h(-m',0))=h(0,m')$.  Now let $f: M \to \L^*$ be the homomorphism defined by $f(m,m')=h(-m,0)$. Then, we have:
$$f^{-1}(m,m') (\gamma \star f )(m,m')  = h(m,0) \varphi_{\gamma}(h(-m', 0))  =    h(m,0)  h(0, m')   = h(m, m').$$ 
Hence $a$ is equivalent to the cocycle $\Gamma \to G, \gamma \mapsto 1$, and so $\H^1( \Gamma, G) = \{ 1 \}$.
\end{proof}

\begin{lemma} \label{H1GRWR}
Let $M:=M_1 \oplus M_0 \oplus M_0$, where $M_1 \cong \Z^p$ and $M_0 \cong \Z^q$,  and let $\rho := \rho_1 \times \rho_0 : \Gamma \to \GL(M)$ be the representation such that $\rho_1$ is the identity on $M_1$ and $\rho_0$   exchanges the two factors on $M_0 \oplus M_0$.  Let $G= \Hom_{gr}(M, \L^* )$ be endowed with the $\Gamma$-action defined by $\gamma \star f := \varphi_{\gamma} \circ f \circ \rho(\gamma^{-1})$.
Then $\H^1( \Gamma, {G}) = \{ 1 \}$.	
\end{lemma}

\begin{proof} 
Denote   $G_1:= \Hom_{gr}(M_1, \L^*)$ and    $G_0:= \Hom_{gr}(M_0 \oplus M_0, \L^*)$.    
We have  a $\Gamma$-equivariant short exact sequence of groups:  
\begin{center}
\begin{tikzpicture}
\matrix (m) [matrix of math nodes,row sep=0.7em,column sep=4em,minimum width=2em]  {
	1_{\vphantom{n}}   & G_1   & G_{\vphantom{n}} & G_{0} & 1_{\vphantom{n}}    \\};
 \path[-stealth]
  	(m-1-1) edge   (m-1-2)
		(m-1-2) edge   (m-1-3)
		(m-1-3) edge   (m-1-4)
		(m-1-4) edge   (m-1-5);
	\end{tikzpicture}
\end{center}
We obtain an exact sequence  in Galois cohomology: 
\begin{center}
\begin{tikzpicture}
\matrix (m) [matrix of math nodes,row sep=0.7em,column sep=2em,minimum width=2em]  {
1 & G_1^{\Gamma}   & G_{\vphantom{n}}^{\Gamma} & G_{0}^{\Gamma} &  \H^1( \Gamma, G_1) & \H^1( \Gamma, G) & \H^1( \Gamma, G_{0})   \\};
 \path[-stealth]
  	(m-1-1) edge    (m-1-2)
		(m-1-2) edge    (m-1-3)
		(m-1-3) edge    (m-1-4)
		(m-1-4) edge    (m-1-5)
		(m-1-5) edge    (m-1-6)
		(m-1-6) edge    (m-1-7);
	\end{tikzpicture}
\end{center}
By Lemma \ref{H1} and Lemma \ref{H1WR}, we have $\H^1( \Gamma, G_{1})  = \{ 1 \}$ and  $\H^1( \Gamma, G_{q})  = \{ 1 \}$. Hence 
$$\H^1( \Gamma, G)  = \{ 1 \}.$$ 
\end{proof}

A consequence of the Lemma \ref{H1GRWR} is the following proposition:

\begin{proposition}  \label{ResultQS} Fix a real torus $(\GC^{n}, \tau)$ where $\tau = \tau_0^p \times \tau_2^q$ and $n = p + 2q$. Let $M:=\Hom_{gr}(\GC^n, \GC)$. 
\begin{enumerate}[leftmargin=0.55cm, label=(\roman*)]
\item    Let $ (Y, \sigma_Y) $ be a normal semi-projective variety. 
 Let $\omega_{N}$ be a pointed {cone} in $N_{\Q}$ and $\D$ be an $\omega_{N}$-pp divisor on $Y$.    
If 
\begin{equation} \label{eqdivquasisplit} 
\let\veqno\eqno
\forall m \in \omega_{N}^{\vee} \cap M, \  \sigma_Y^*(\D(m)) = \D(\tilde{\tau}(m)),
\end{equation}
then there exists an $\R$-structure $\sigma_{X[Y, \D]}$ on the affine variety $X[Y, \D]$ such that  $(\GC^n, \tau)$ acts  on $(X[Y, \D], \sigma_{X[Y, \D]})$.  
\item  Let $(X, \sigma_X)$ be an affine variety endowed with an action of $(\GC^{n}  , \tau)$ of weight cone $\omega_{M} \subset M_{\Q}$, and let $\omega_{N}$ be the cone in ${N}$ dual to $ \omega_{M}$.    There exists a semi-projective variety $(Y, \sigma_Y)$ and an $\omega_{N}$-pp divisor  $\D$  on $Y$ 
 such that  
$$\forall  m \in \omega_{M} \cap M, \ \sigma_Y^*(\D(m)) = \D(\tilde{\tau}(m)),$$
and such that the varieties $(X, \sigma_X)$ and $({X[Y, \D']}, \sigma_{X[Y, \D']})$ are $(\GC^{n}, \tau)$-equivariantly isomorphic.
\end{enumerate}
\end{proposition}

For a  complexity one quasi-split $(\T, \tau)$-action  on an affine $\R$-variety $(X, \sigma)$ (i.e~$\text{Dim}(\T)=\text{Dim}(X)-1$), we recover the real version of Langlois' result in \cite{Lang} about quasi-split torus actions on varieties of complexity one  over an arbitrary field.

\subsection{Some one-to-one correspondence} \label{SubsectionFH}

The two main results (Theorem \ref{ResultT} and \ref{ResultT2}) establish correspondences between real affine  varieties endowed with a real torus action and triples $(Y, \D, h)$. In this section,  we focus on this correspondence.  In general, there is no one-to-one correspondence between $(\T, \tau)$-varieties and triple $(Y, \D, h)$. However, Altmann and Hausen define the notion of minimal pp-divisor in \cite[Section 8]{Alt} that leads us to the following result:

\begin{theorem} \label{ThmFunct} Let $(\T, \tau)$ be a real torus and let $M:= \Hom_{gr}(\T, \GC)$. Let $\omega_N \subset N_{\Q}$ (resp. $\omega_N'\subset N_{\Q}$) be  a pointed cone, let $(Y, \sigma_Y)$ (resp. $(Y', \sigma_Y')$) be a normal semi projective variety,    $\D \in \PPDiv(Y, \omega_N)$ (resp.  $\D' \in \PPDiv(Y', \omega_N')$) be a minimal pp-divisor   and   $h : \omega_N^{\vee} \cap M \to \C(Y)$  be a monoid morphism such that:
$$\forall  m \in \omega_{N}^{\vee} \cap M, \ \sigma_Y^*(\D(m)) = \D(\tilde{\tau}(m)) + \Div_Y(h(\tilde{\tau}(m))) \text{ and } h(m) \sigma_Y^{\sharp}(h(\tilde{\tau}(m)))=1,$$
(resp. $h': {\omega_N}'^{\vee} \cap M \to \C(Y')$). The affine $\R$-varieties $(X[Y, \D], \sigma_{X[Y, \D]})$ and $(X[Y', \D'], \sigma_{X[Y', \D']})$ are $(\T, \tau)$-isomorphic if and only if the following holds:
\begin{enumerate}[leftmargin=0.75cm, label=(\roman*)]
\item there exists an isomorphism $\psi : (Y, \sigma_Y) \to (Y', \sigma_Y')$;
\item there exists a lattice automorphism $L : N \to N $  such that $L \circ \hat{\tau} = \hat{\tau} \circ L$;
\item there exists a monoid morphism $g : M \to \C(Y)$;
\item for all $m \in \omega_M \cap M$, $\psi^*(\D'(m))= \D(L^*(m)) + \Div_Y(g(m))$;
\item for all $m \in \omega_M \cap M$, $\dfrac{{\sigma_Y}^{\sharp}(g(\tilde{\tau}(m)))}{g(m)}=\dfrac{\psi^{\sharp}(h(m))}{h'(L^*(m))}$ (i.e the cocycles defined by $h' \circ L^*$ and $\psi^{\sharp} \circ h$ are equivalent).
\end{enumerate}
\end{theorem}

\begin{proof}
By \cite[Theorem 8.8]{Alt}, the affine $\C$-varieties $X[Y, \D]$ and $X[Y', \D']$ are $\T$-isomorphic if and only if the following holds:
\begin{enumerate}[leftmargin=0.75cm, label=(\roman*)]
\item there exists an isomorphism $\psi : Y' \to Y$;
\item there exists a lattice automorphism $L : N \to N $;
\item there exists a monoid morphism $g : M \to \C(Y')$;
\item for all $m \in \omega_M \cap M$, $\psi^*(\D'(m))= \D(L^*(m)) + \Div_Y(g(m))$; 
\end{enumerate}
The    morphisms $\psi$ and $L$ induces a $\T$-equivariant isomorphism of graded algebras:
$$ \Psi : A[Y, \D] \to A[Y', D'], f \mathfrak{X}_{m}  \mapsto \psi^{\sharp}(f) g(m)\mathfrak{X}_{L^*(m)}.$$ 
Therefore the  diagram
\begin{center}
\begin{tikzpicture}
  \matrix (m) [matrix of math nodes,row sep=1em,column sep=4em,minimum width=1.25em]
  {
	   A[Y, \D] &   A[Y', \D']  \\
	   A[Y, \D] &   A[Y', \D']  \\};
 \path[-stealth]
   	(m-1-1) edge   node [above] { $\Psi$}   (m-1-2)
		(m-1-1) edge   node [ left] { $ \sigma_Y^{\sharp} $ }  (m-2-1)
		(m-1-2) edge    node [ left] { ${\sigma_Y^{\sharp}}' $ }   (m-2-2)
		(m-2-1) edge    node [above] { $ \Psi $ }  (m-2-2);
	\end{tikzpicture}
\end{center}
commutes if and only for all $m \in \omega_M \cap M$, $\dfrac{{\sigma_Y}^{\sharp}(g(\tilde{\tau}(m)))}{g(m)}=\dfrac{\psi^{\sharp}(h(m))}{h'(L^*(m))}$. 
\end{proof}

\section{Examples} \label{SectionEx}

\subsection{Split real torus actions on  normal affine $\R$-varieties} \label{SubsectionGR}  

We describe actions of the real split torus $\GR^n$ on   affine $\R$-varieties. By definition, a $\GR^n$-action on an $\R$-variety $(X, \sigma_X)$ is an action of the real torus $(\GC^n, \tau_0^{\times n})$ on $(X, \sigma_X)$.  Fix a real torus $\GR^n = (\GC^n, \tau_0^{\times n})$ and let $M:=\Hom_{gr}(\GC^n, \GC)$.  The condition (\ref{eqdivquasisplit}) of Proposition  \ref{ResultQS} becomes:
$$\forall  m \in \omega_{M} \cap M, \ {\sigma_Y}^*(\D(m)) = \D(m).$$

\begin{example} \label{exGMR} We pursue Example \ref{exGMRa}.  In the case of a $\GR$-action, the sequence obtained from the inclusion $(X, \sigma) \hookrightarrow (\AC^n, \sigma')$ of Proposition \ref{GaloisToricDowngrading}  does not always have a $\Gamma$-equivariant section.  Indeed, consider the affine variety $(\AC^2, \sigma)$, where $\sigma(x,y) = (\overline{y}, \overline{x})$. Note that the torus $(\GC, \tau_0)$ acts on $(\AC^2, \sigma)$ by $ t \cdot (x,y)   = (tx, ty)$. Then, we have the following split short exact sequence
\begin{center}
\begin{tikzpicture}
\matrix (m) [matrix of math nodes,row sep=0.7em,column sep=5em,minimum width=2em]  {
	1^{\vphantom{n}}_{\vphantom{n}}   & \GC   & \GC^2 & \GC  & 1    \\};
 \path[-stealth]
  	(m-1-1) edge    (m-1-2)
		(m-1-2) edge    (m-1-3)
		(m-1-3) edge    (m-1-4)
		(m-1-4) edge    (m-1-5);
	\end{tikzpicture}
\end{center}
with $\GC \to \GC^2, t \mapsto (t,t)$ and $\GC^2 \to \GC, (s,t) \mapsto s/t$. 
We obtain the diagrams of Remark \ref{FracsectionR} with $M'=\Z^2$, $M=\Z$, $M_Y=\Z$, and with the following lattice homomorphisms:
\begin{multicols}{5}
 \begin{center}
{\ \\
\vspace{-0.25cm} 
$F^* := \begin{bmatrix}
1 & 1   \\
\end{bmatrix}$;
\ \\}
\end{center} 

\begin{center}
{
$P^* := \begin{bmatrix}
 1   \\
-1   \\
\end{bmatrix}$};
\end{center}

\begin{center}
{\vspace{-0.22cm}
$\tilde{\tau}' := \begin{bmatrix}
0 & 1 \\
1 & 0 \\
\end{bmatrix}$};
\ \\
\end{center}

\begin{center}
{ \ \\
\vspace{-0.25cm}  
$\tilde{\tau}_Y := \begin{bmatrix}
-1   \\
\end{bmatrix}$; 
\ \\}
\end{center}

\begin{center}
{ \ \\
\vspace{-0.25cm} 
$\tilde{\tau} = \begin{bmatrix}
1   \\
\end{bmatrix}$.
\ \\}
\end{center}
\end{multicols}
We can  show that there  is no   $\Gamma$-equivariant section $s^*: M  \to \Z^2$. Indeed, note that if such a section exists, we obtain $\GR \times \mathbb{S}^1 \cong \RC(\GC) $, which is false (see Proposition \ref{RealGroupStructureTorus}). Let   
$$s^* := \begin{bmatrix}
0   \\
1   \\
\end{bmatrix}$$
 be a section. Then, an Altmann-Hausen  presentation   of the $\GR$-action on $(\AC^2, \sigma)$ is:
\begin{multicols}{2}
\begin{itemize}
\item $\mathbb{P}^1 = {\AC^1}_v \cup \{ \infty \} $;   
\item $\sigma_Y : v \mapsto \overline{v}^{\ -1} $;  
\end{itemize}

\begin{itemize}
\item $\D  :=  [ 1 , + \infty [  \otimes \{  \infty \}$;  
\item $h := w = x/y$.
\end{itemize}
\end{multicols}

Now we give a $(\Gamma \times \GC)$-equivariant inclusion $(X, \sigma) \hookrightarrow (\AC^n, \sigma')$ which admits a $\Gamma$-equivariant section. First, note that $\AC^2 \cong \Spec(\C[x,y,z]/(x+y-z)) \subset \AC^3$, where the closed embedding is given by:
\begin{equation*}
\AC^2 \to \AC^3, \ (x,y) \mapsto (x,y,x+y).
\end{equation*}
Consider the action of $\GC$ on $\AC^3$ defined by $t \cdot (x,y,z) = (tx, ty, tz)$. This action is obtained from the inclusion:
\begin{equation*}
\GC \to \GC^3, \ t \mapsto (t,t,t).
\end{equation*}
Consider the $\R$-group structure on $\GC^3$ defined by $\sigma'(t_1,t_2,t_3) = (\overline{t_2}, \overline{t_1}, \overline{t_3})$.  The closed immersion $\AC^2 \cong \Spec(\C[x,y,z]/(x+y-z)) \subset \AC^3$ is $(\Gamma \times \GC)$-equivariant. 
We obtain the diagrams of Remark \ref{FracsectionR} with $M'=\Z^3$, $M=Z$, $M_Y=\Z^2$, and with the following lattice homomorphisms:
\begin{multicols}{5}
\footnotesize{
 \begin{center}
{ \ \\ 
\vspace{-0.1cm}
$F^*:= \begin{bmatrix}
1 & 1 & 1 \\
\end{bmatrix}$;
\ \\ }
\end{center} 

\begin{center}
$P^* := \begin{bmatrix}
 1  &  0  \\
 0  &  1  \\
-1  & -1  \\
\end{bmatrix}$;
\end{center}

\begin{center}
\vspace{-0.22cm}
$\tilde{\tau}' := \begin{bmatrix}
0 & 1 & 0 \\
1 & 0 & 0 \\
0 & 0 & 1 \\
\end{bmatrix}$;
\end{center}

\begin{center}
{\ \\
\vspace{-0.25cm}
$\tilde{\tau}_Y := \begin{bmatrix}
0 & 1   \\
1 & 0   \\
\end{bmatrix}$;
\ \\}
\end{center}

\begin{center}
$s^* := \begin{bmatrix}
0 \\
0 \\
1 \\
\end{bmatrix}$.
\end{center}}
\end{multicols} 
The section $s^*$ is $\Gamma$-equivariant. An Altmann-Hausen  presentation of the $\GR$-action on $(\AC^3, \sigma')$  is given by: 
\begin{itemize}[label=$\bullet$]
\item ${Y } := \mathbb{P}^2  =  U_1 \cup U_2 \cup U_3$, where 
$$U_1= \Spec(\C[v_1,w_1]), \  U_2= \Spec(\C[v_2,w_2]),  \text{ and }  U_3=  \Spec(\C[v_3,w_3]),$$ with gluing morphism obtained from $(v_1=x/z,w_1=y/z)$,  $(v_2=y/x, w_2=z/x)$ and $(v_3=z/y, w_3=x/y)$;  
\item $\sigma_{Y}$ is the $\R$-structure exchanging $x$ and  $y$ and fixing $z$;   
\item $\D:=  [ 1 , + \infty [  \otimes D$,  where $D \big{|}_{U_1} = 0$, $D \big{|}_{U_2} = \{ w_2 = 0\}     $ and $D \big{|}_{U_3} = \{ v_3 = 0\}     $;   and  
\item $h  := 1$.
\end{itemize}
We deduce from this an Altmann-Hausen  presentation of the $\GR$-action on $(\Spec(\C[x,y,z]/(x+y-z)), \sigma')$:
\begin{itemize}
\item ${Y_X} :=  {U_1}_X \cup {U_2}_X \cup {U_3}_X \cong \mathbb{P}^1$, where
$${U_1}_X= \Spec(\C[v_1,w_1]/(v_1+w_1-1)),  \   {U_2}_X= \Spec(\C[v_2,w_2]) /(v_2 - w_2  +1 ),  $$ 
 $$\text{ and } {U_3}_X=  \Spec(\C[v_3,w_3] / (v_3  - w_3 -1  ));$$  
\item $\sigma_{Y_X} = \sigma_{Y} \big{|}_X$;   
\item $\D_X  :=  [1,+ \infty [ \otimes D$,  where $D \big{|}_{U_1} = 0$, $D \big{|}_{U_2} = \{ (1,0) \}     $ and $D \big{|}_{U_3} = \{ (0,1) \}     $; and  
\item $h_X :=1$.
\end{itemize}
\end{example}

\subsection{Weil restriction   actions on  normal affine $\R$-varieties} \label{SubsectionWR}

By definition, a $\RC(\GC)$-action on a real algebraic variety $(X, \sigma_X)$ is an action of the real torus $(\GC, \tau_2)$ on $(X, \sigma_X)$.  Fix a real torus $(\GC^2, \tau_2)$. Let   $M:=\Hom_{gr}(\GC^{2}, \GC)$ and $N$ its dual lattice. 
The condition (\ref{eqdivquasisplit}) of Proposition  \ref{ResultQS} becomes:
$$\forall  m \in \omega_{M} \cap M, \ {\sigma_Y}^*(\D(m)) = \D(\tilde{\tau}_2(m)).$$

\begin{example}   The sequences of Example \ref{RWA3''} admit a $\Gamma$-equivariant section $s : N' \to N$ defined by 
\begin{equation*}
 s:= \begin{bmatrix}
1 & 0 & 0 \\
0 & 1 & 0 \\
\end{bmatrix}
\end{equation*}
 An Altmann-Hausen  presentation of the $(\GC^2, \tau_2)$-action on $(\AC^3, \sigma')$ is given by 
\begin{equation*}
((Y, \sigma_Y), \D, h),
\end{equation*}
where (see Example  \ref{RWA3''}):
\begin{itemize}
\item $Y := \mathbb{P}^1 = {\AC^1} \cup \{ \infty \} $;   
\item $\sigma_Y$ is the complex conjugation on the coordinates;
\item $\D  :=  \Delta \otimes \{ \infty \}$; and   
\item $h := 1$.
\end{itemize}
\end{example}

\begin{example}  \label{RWA4''}
We give an Altmann-Hausen presentation of the $\GC^2$-action on $\AC^4$ introduced in Example \ref{RWA4}.   Using toric downgrading results of  \cite[\S 11]{Alt}, we  obtain  the presentation of the $\GC^2$-action on $X$.
Consider the immersion  $\T:= \GC^2 \hookrightarrow \GC^4$, $(s,t) \mapsto (s,t,st^2, s^2t)$. We denote  $\T_Y:= \GC^4/\T$,  $M:=\Hom_{gr}(\T, \GC)$, $M':=\Hom_{gr}(\GC^4, \GC)$ and   $M_Y:=\Hom_{gr}(\T_Y, \GC)$. Then, we have the split short exact sequences of Remark \ref{FracsectionR} with:
\begin{multicols}{4}
 \begin{center}
{ \ \\
$F^* := \begin{bmatrix}
1 & 0 & 1 & 2   \\
0 & 1 & 2 & 1   \\
\end{bmatrix}$;
\ \\}
\end{center} 

\begin{center}
{$P^* := \begin{bmatrix}
 -1  & -2  \\
-2 & -1   \\
1 & 0   \\
0 & 1   \\
\end{bmatrix}$};
\end{center}

\begin{center}
{$\tilde{\tau}' := \begin{bmatrix}
0 & 1 & 0 & 0 \\
1 & 0 & 0 & 0\\
0 & 0 & 0 & 1 \\
0 & 0 & 1 & 0 \\
\end{bmatrix}$;
\ \\}
\end{center}

\begin{center}
{ \ \\ 
$\tilde{\tau}_Y = \tilde{\tau}  := \begin{bmatrix}
0 & 1 \\
1 & 0  \\
\end{bmatrix}$.
\ \\}
\end{center}
\end{multicols}
The section $s:N' \to N$ defined by
{$s := \begin{bmatrix}
1 & 0 & 0 & 0 \\
0 & 1 & 0 & 0\\
\end{bmatrix}$}
is $\Gamma$-equivariant. Let $Y$ be the toric variety defined by the following fan  obtained from $P$ (see \cite[\S 11]{Alt} for details):
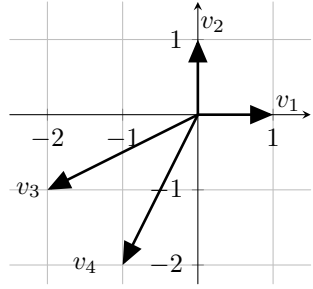
\begin{center}
\begin{tikzpicture}[line cap=round,line join=round,>=triangle 45,x=1.0cm,y=1.0cm]
\begin{axis}[
x=1.0cm,y=1.0cm,
axis lines=middle,
ymajorgrids=true,
xmajorgrids=true,
xmin=-2.5,
xmax=1.5,
ymin=-2.25,
ymax=1.5,
xtick={-3.0,-2.0,...,3.0},
ytick={-3.0,-2.0,...,3.0},]
\clip(-3.,-3.) rectangle (3.,3.);
\draw [->,line width=1.pt] (0.,0.) -- (1.,0.);
\draw [->,line width=1.pt] (0.,0.) -- (0.,1.);
\draw [->,line width=1.pt] (0.,0.) -- (-2.,-1.);
\draw [->,line width=1.pt] (0.,0.) -- (-1.,-2.);
\begin{scriptsize}
\draw[color=black] (1.2,0.15) node {$v_1$};
\draw[color=black] (0.2,1.2) node {$v_2$};
\draw[color=black] (-2.25,-1) node {$v_3$};
\draw[color=black] (-1.5,-2) node {$v_4$};
\end{scriptsize}
\end{axis}
\end{tikzpicture}
\end{center}
Since this fan is  stable under  the lattice involution $\hat{\tau}_Y$, the $\R$-group  structure $\tau_Y$ extends to an $\R$-structure $\sigma_Y$ on $Y$. 
Let $\D$ be the divisor defined in Example \ref{RWA4'}, where $D_1, \dots, D_4$ are the toric divisors obtained from the rays $v_1, \dots, v_4$ respectively.

An Altmann-Hausen  presentation of the $(\GC^2, \tau_2)$-action on $(\AC^4, \sigma')$ is given by $((Y, \sigma_Y), \D, h=1)$.
An  Altmann-Hausen  presentation of the $(\GC^2, \tau_2)$-action on $(X, \sigma)$ is given by $((Y_X, \sigma_{Y_X}), \D_X, h_X)$, where:
\begin{itemize}
\item $Y_X $ is the  closure of  the image of $\GC^4 \cap X  $ in $Y$; 
\item $\D_X  :=  \Delta_3 \otimes D_3 \cap Y_{X} +  \Delta_4 \otimes D_4 \cap Y_{X} $; 
\item  $\sigma_{Y_X} = \sigma_{Y} \big{|}_{Y_X}$; and 
\item $h_X := 1$.
\end{itemize}
\end{example}

\subsection{Circles actions  on normal affine $\R$-varieties} \label{SubsectionC}

By definition, a $\s^1$-action on a algebraic $\R$-variety $(X, \sigma_X)$ is an action of the  torus $(\GC, \tau_1)$ on $(X, \sigma_X)$. Note that $\GC$ acts on $X$ and  the algebra $\C[X]$ is  graded by $M\cong \Z$. By  \cite[Lemma 1.7]{Lien},    $\C[X]_m \neq 0$ for all $m \in M$, so the weight cone of this action is $\omega_{M}:= M_{\Q}$.

In this case, the   pair $(\D, h)$ on the  quotient $(Y, \sigma_Y)$ mentioned in Theorem \ref{ResultT} \textit{(ii)} consists of a proper polyhedral divisor $\D = \sum [ a_i, b_i ] \otimes D_i$ and a $\Gamma$-invariant rational function $h$ on $Y$   such that $\sigma_Y^*( \D(m)) = \D(-m) + \Div_Y(h^{-m})$ for all $m \in M$ (we recover \cite[Theorem 2.7]{Lien}).

In the case of $\s^1$-actions, we do not have $\H^1(\Gamma, \C(Y)^*)=\{1 \}$. Indeed, in contrast to the split case, we cannot apply Hilbert's theorem 90   because the action defined on $\C(Y)^*$ does not extend to an action on the field $\C(Y)$ (see the proof of Lemma \ref{H1}).

Let $Y= \Spec(\C)$   endowed with the complex conjugation, and let $G:= \Hom_{gr}(\Z, \C(Y)^*) \cong \Hom_{gr}(\Z, \C^*) \cong \C^*$. The $\Gamma$-action on $G \cong \C^*$ is given by $\gamma \star z := \bar{z}^{ \ -1}$. A cocycle is thus a complex number $z \in \C^*$ such that $z \bar{z}^{ \ -1}= 1$, that is a real number. This cocycle is equivalent to $1 \in \C^*$ if there exists $w \in G \cong \C^*$ such that $z = w^{-1} \bar{w}^{ \ -1} = \rvert w \lvert^{-2} > 0$.   Then, $\H^1(\Gamma, \C^* ) \cong \{ \pm 1\}$.

\begin{example} \label{exc} (See \cite[Proposition 3.1]{Lien}). 
There are only  two $\R$-forms of $\GC$ compatible with an $\s^1$-action:   the real circle $X_{1} = \s^1$ and $X_{-1} = \Spec(\R[x,y]/(x^2+y^2+1))$.  An $\R$-structure  associated to $X_{1}$ is $\sigma_{1} := \tau_{1} : \GC \to \GC, z \mapsto  \overline{z}^{ \ -1 }$ and an $\R$-structure  associated to $X_{-1}$ is $\sigma_{-1} : \GC \to \GC, z \mapsto -\overline{z}^{ \ -1 }$. Consider the action by translation $\mu : \GC \times \GC \to \GC, (t,x) \mapsto tx$. The varieties $X_{1}$ and $X_{-1}$ are   endowed with an $\s^1$-action since the  following diagram  commutes for $i \in \{-1, 1 \}$:
\begin{center}
\begin{tikzpicture}
\matrix (m) [matrix of math nodes,row sep=0.75em,column sep=6em,minimum width=2em]
 {
          \GC \times \GC & \GC  \\
          \GC \times \GC & \GC   \\};
 \path[-stealth]
    (m-1-1) edge node  [above]{ $\mu $} (m-1-2)
            edge node  [left]{${\tau_1}  \times {\sigma_{i} }  $} (m-2-1)
    (m-1-2) edge node  [right]{${\sigma_{i} } $}  (m-2-2)
    (m-2-1) edge node  [below]{$\mu$} (m-2-2);
\end{tikzpicture}
\end{center}
The $\R$-variety $(Y=\Spec(\C), \sigma_Y)$, where $\sigma_Y$  is the complex conjugation, is a real Altmann-Hausen quotient of both $X_{1}$ and $X_{-1}$. 

A  pair $(\D_{1},h_{1})$ on $Y$ consists of the trivial divisor and the real number $h_{1}=1 \in \C(Y)^*=\C^*$.  Note however that $\H^1(\Gamma, \C^* ) \neq \{   1\}$.

A  pair $(\D_{-1},h_{-1})$ on $Y$ consists of the trivial divisor and the real number $h_{-1}=-1 \in \C(Y)^*=\C^*$. Note that we cannot find a complex number $g \in \C(Y)^*$ such that the cocycle $h_{-1}$ satisfy $h_{-1}=g \bar{g}$, so we cannot find a presentation where the $\Gamma$-invariant rational function $h$ mentioned in Theorem \ref{ResultT2} is equal to $1$.   
\end{example}

\bibliographystyle{plain}

\end{document}